\documentclass[10pt, english]{amsart}

\usepackage{amsmath,amssymb,enumerate}

\usepackage[T1]{fontenc}
\usepackage[all]{xy}

\usepackage{babel}
\usepackage{amstext}
\usepackage{amsmath}
\usepackage{amsfonts}
\usepackage{mathtools}
\usepackage{latexsym}
\usepackage{ifthen}
\usepackage{xcolor}
\usepackage{tikz-cd}

\usepackage{xypic}
\xyoption{all}
\newtheorem{lemma1}{}[section]

\newenvironment{lemma}{\begin{lemma1}{\bf Lemma.}}{\end{lemma1}}

\newenvironment{theorem}{\begin{lemma1}{\bf Theorem.}}{\end{lemma1}}
\newenvironment{proposition}{\begin{lemma1}{\bf Proposition.}}{\end{lemma1}}
\newenvironment{corollary}{\begin{lemma1}{\bf Corollary.}}{\end{lemma1}}
\newenvironment{remark}{\begin{lemma1}{\bf Remark.}\rm}{\end{lemma1}}

\newenvironment{definition}{\begin{lemma1}{\bf Definition.}}{\end{lemma1}}

\newenvironment{setup}{\begin{lemma1}{\bf Set-up.}}{\end{lemma1}}

\newenvironment{conjecture}{\begin {lemma1}{\bf Conjecture.}}{\end{lemma1}}

\newenvironment{terminology}{\begin{lemma1}{\bf Terminology.}}{\end{lemma1}}

\newenvironment{remark*}{{\bf Remark.}}{}
\newenvironment{example*}{{\bf Example.}}{}
\newenvironment{assumption*}{{\bf Assumption.}}{}

\newcommand{\R}{\ensuremath{\mathbb{R}}}
\newcommand{\Q}{\ensuremath{\mathbb{Q}}}
\newcommand{\Z}{\ensuremath{\mathbb{Z}}}
\newcommand{\C}{\ensuremath{\mathbb{C}}}
\newcommand{\N}{\ensuremath{\mathbb{N}}}
\newcommand{\PP}{\ensuremath{\mathbb{P}}}

\newcommand{\merom}[3]{\ensuremath{#1\colon #2 \dashrightarrow #3}}

\newcommand{\holom}[3]{\ensuremath{#1\colon #2  \rightarrow #3}}

\makeatletter
\ifnum\@ptsize=0  \fi \ifnum\@ptsize=2  \fi 

\newcommand\sO{{\mathcal O}}

\DeclareMathOperator*{\pic}{Pic}

\newcommand{\Chow}[1]{\ensuremath{\mbox{\rm Chow}(#1)}}

\newcommand{\NEX}{\overline{\mbox{NE}}(X)}

\newcommand{\NE}[1]{ \ensuremath{ \overline { \mbox{NE} }(#1)} }

\DeclareMathOperator*{\NS}{NS}

\newcommand{\Mov}[1]{ \ensuremath{ \overline{ \mbox{\rm Mov} }(#1)} }
\newcommand{\Move}[1]{ \ensuremath{ \overline{ \mbox{\rm Mov} }^e(#1)} }
\newcommand{\MovO}[1]{ \ensuremath{ \mbox{\rm Mov}^\circ(#1)} }
\newcommand{\Movp}[1]{ \ensuremath{ 
{ \mbox{\rm Mov} }^+(#1)} }

\newcommand{\Nef}[1]{ \ensuremath{ \mbox{\rm Nef}(#1)} }
\newcommand{\Nefe}[1]{ \ensuremath{ \mbox{\rm Nef}^e(#1)} }
\newcommand{\Nefp}[1]{ \ensuremath{  \mbox{\rm Nef}^+(#1)} }

\newcommand{\Amp}[1]{ \ensuremath{ \mbox{\rm Amp}(#1)} }

\newcommand{\Eff}[1]{ \ensuremath{ \mbox{\rm Eff}(#1)} }
\newcommand{\EffO}[1]{ \ensuremath{ \mbox{\rm Eff}^\circ(#1)} }

\newcommand{\Bigcone}[1]{ \ensuremath{ \mbox{\rm Big}(#1)} }

\newcommand{\ClQ}[1]{ \ensuremath{ \mbox{Cl}(#1)_{\Q}} }

\DeclareMathOperator*{\Aut}{\rm Aut}
\DeclareMathOperator*{\Bir}{\rm Bir}
\DeclareMathOperator*{\PsAut}{\rm PsAut}
\DeclareMathOperator{\Exc}{\rm Exc}

\usepackage{url}

\numberwithin{equation}{section}

\title{On the relative cone conjecture for families of IHS manifolds} 
\date{\today}


\author{Andreas H\"oring}
\author{Gianluca Pacienza}
\author{Zhixin Xie}

\address{Andreas H\"oring, Universit\'e C\^ote d'Azur, CNRS, LJAD, France}
\email{Andreas.Hoering@univ-cotedazur.fr}
\address{Gianluca Pacienza,
Universit\'e de Lorraine, CNRS, IECL, 
F-54000 Nancy, France}
\email{gianluca.pacienza@univ-lorraine.fr}

\address{Zhixin Xie,
Universit\'e de Lorraine, CNRS, IECL, 
F-54000 Nancy, France}
\email{zhixin.xie@univ-lorraine.fr}

\begin{document}

\begin{abstract} 
We study the relative cone conjecture for families of $K$-trivial varieties 
with vanishing irregularity. As an application 
we prove that  the relative movable and the relative nef cone conjectures hold for fibrations in projective IHS manifolds of the 4 known deformation types. 
\end{abstract}

\maketitle

\section{Introduction}

By the cone theorem the Mori cone of a Fano manifold is always a rational polyhedral cone.
This no longer holds for projective manifolds with trivial canonical class
which may have a nef cone with ``round'' shape or infinitely many extremal rays \cite{Kov94, Bau98}.
The Kawamata-Morrison cone conjecture ({see \cite{Mor93, Mor96, Kaw97}) 
predicts that in this case the action of the automorphism group on the cone of effective nef classes  admits a rational, polyhedral fundamental domain.  To state the conjecture more precisely and more generally (following Kawamata \cite{Kaw97} and Totaro \cite{Tot10}),  we formalise the  set-up. 
\begin{definition}
A klt K-trivial fibration is a normal $\Q$-factorial klt pair $(X, B)$,
endowed with a fibration 
$
\holom{\pi}{X}{S}
$ 
over a quasi-projective variety $S$
such that $K_X+B \equiv_\pi 0$.
\end{definition}

\begin{remark*}
Since the nonvanishing conjecture is known for klt pairs that are numerically trivial, e.g. \cite[Theorem 3.1]{FG12}, the assumption implies $K_X+B \sim_{\R, \pi} 0$ 
\end{remark*}

\smallskip
The  Kawamata-Morrison cone conjecture, with the extension to pairs introduced by Totaro, is the following.
\begin{conjecture}\label{conj:cone}
Let $(X,B)\to S$ be a klt K-trivial fibration. Then the following hold.
\begin{enumerate}
\item[(1)] 
There exists a rational polyhedral  cone $\Pi$ which is a fundamental domain for the
action of the group of relative automorphisms preserving the boundary $\Aut(X/S,B)$ on 
$$
\Nefe{X/S}\coloneqq\Nef{X/S}\cap \Eff{X/S}
$$
in the sense that
\begin{enumerate}
\item[(a)] $\Nefe{X/S}= \cup_{g\in \Aut(X/S,B)} g^*\Pi$.
\item[(b)] $\Pi^\circ \cap g^*\Pi^\circ = \emptyset$, unless $g^* = \textrm{id}$ in $\textrm{GL}\big(N^1(X/S)_\R\big)$.
\end{enumerate}
\item[(2)] 
There exists a rational polyhedral  cone $\Pi'$ which is a fundamental domain in the sense above for the
action of the group of relative pseudoautomorphisms preserving the boundary $\PsAut(X/S,B)$ on  
$$
\Move{X/S}\coloneqq\Mov{X/S}\cap \Eff{X/S}.
$$ 
\end{enumerate}
\end{conjecture}

Item (1) (resp.\ (2)) above will be referred to as the relative nef (resp. movable) cone conjecture for $(X,B)$ over $S$. 
We refer the reader to Section \ref{subsec:cones} for all the definitions and variants of the above conjecture whose relevance, apart from providing a unified framework for Fano and $K$-trivial varieties, is due to its consequences for the birational geometry of a $K$-trivial variety $X$ (see e.g.\ \cite[Theorem 1.5]{GLSW24}). 

Conjecture \ref{conj:cone} is open even in the smooth absolute case without boundary, i.e.\ when $S$ is a point, $X$ is smooth and $B=0$. In this case we know by the celebrated Beauville-Bogomolov decomposition theorem that a projective $K$-trivial manifold is, up to a finite \'etale cover, a product of Calabi-Yau (CY) manifolds, of irreducible holomorphic symplectic (IHS) manifolds and of an abelian variety. 
Although Prendergast-Smith \cite{P-S12} proved the conjecture for abelian varieties and Markman \cite{Mar11} and Amerik-Verbitsky  \cite{AV17, AV20} proved it for IHS manifolds, there is no general descent result for finite étale covers (see however \cite{PS23} and \cite{MQ24})) 
and for CY manifolds, despite the effort of many people, it is wide open already in dimension 3 (cf.\ \cite{Lut24}). 

In the relative case the situation is even worse and no results at all are known for fibrations having arbitrarily high relative dimension. After the pioneering work of Kawamata in dimension 3 \cite{Kaw97}, only recently weak versions of the relative cone conjecture were obtained for fibrations in surfaces, see \cite{LZ22,Li23,MS24}.

One strong motivation for studying the relative case comes from the fact that, modulo standard conjectures in the MMP, the relative cone conjecture implies for instance that the number of minimal models of any terminal projective variety with non-negative Kodaira dimension is finite up to isomorphisms, see \cite[Theorem 2.14]{CL14}. Notice however that the relative automorphism group and the relative nef effective cone of a fibration $X \rightarrow S$ can be radically different from their counterparts on the very general fibre $X_s$. For example, an automorphism $\alpha_s$ that exists on every general fibre $X_s$ may not extend to $X$, even as a birational automorphism \cite[Example 1.3(1)]{Ogu24}. In Section \ref{sectionfamilyK3} we exhibit a smooth family of Kummer surfaces,
where $\Aut(X/C)$ is finite, while $\Bir(X/C)$ and $\Aut(X_s)$ are infinite.
Thus the ideas and techniques used to treat the absolute case cannot be immediately transposed to tackle the relative case.

The main result of the present paper  concerns the relative cone conjecture for fibrations in IHS manifolds. 

\begin{theorem}\label{cor:IHS}
Let $\pi\colon X\to S$  be a fibration between $\Q$-factorial normal projective varieties
such that the total space $X$ is klt. Suppose that the very general fibre of $\pi$ is a projective IHS manifold of one of the 4 known deformation types.  
Then the relative movable cone conjecture holds for $X$ and up to isomorphisms over $S$ there exist only finitely many $X'$ arising as small $\Q$-factorial modifications (SQM) of $X$ over $S$, each of which satisfies the  relative nef cone conjecture.
 \end{theorem}
For the definition of SQM we refer the reader to Definition \ref{def:SQM}.
Notice that the geometry of an SQM may be very different from that of the initial fibration $X\to S$ and have, for instance, a relative nef cone with infinitely many effective extremal rays while $\Nefe{X/S}$ is rational polyhedral (see e.g.\ Proposition \ref{prop:interesting} below). Theorem \ref{cor:IHS} above is an immediate consequence of the following. 
  \begin{theorem}\label{thm:IHS-intro-gen}
Let $\pi\colon X\to S$  be a fibration between $\Q$-factorial normal projective varieties
such that the total space $X$ is klt. Suppose that the very general fibre  of $\pi$ is a projective IHS manifold such that every nef divisor is semiample.  Then
\begin{enumerate}[(a)]
 \item the relative movable cone conjecture holds for $X$;
\item up to isomorphism over $S$, there are only finitely many $X'$ arising as small $\Q$-factorial modifications (SQM) of $X$ over $S$  and  the 
        relative nef cone conjecture holds for  each of them.
  \end{enumerate}      
 \end{theorem} 
 Amerik and Verbitsky deduced the nef cone conjecture for IHS manifolds from the movable cone conjecture proved by Markman \cite{Mar11}
 by proving the boundedness of squares of primitive MBM classes
 \cite{AV17, AV20}. 
Our approach, due to \cite{GLSW24} (see also \cite[Theorem 13]{Xu24}), provides a general
and alternative way to derive the nef cone conjecture from the movable cone conjecture
by using the so-called geography of models (cf.\ \cite{SC11}). This requires 
the conjectural assumption on the semiampleness of nef divisors (which was proved by \cite{Mats17, MO22,MR21} for the 4 known deformation types of IHS manifolds, providing in passing a new
proof of the nef cone conjecture for these manifolds), but has the advantage of
working in the relative setting. In particular, Theorem \ref{thm:IHS-intro-gen} follows from a combination of results from \cite{Tak21} and \cite{LZ22} together with a relative version of \cite[Theorem 1.5]{GLSW24}, which is the technical core of this paper and of independent interest. To state it in full generality let us make the framework more precise.

\begin{setup} \label{setup-relcone}
In the rest of the text we will consider a klt $K$-trivial fibration
$
\holom{\pi}{(X,B)}{S}
$ 
such that
\begin{itemize}
\item[-] for a very general fibre $X_s$ its irregularity $q(X_s)$ is zero;
\item[-] the base $S$ is $\Q$-factorial quasi-projective variety.
\end{itemize}
\end{setup}

\begin{remark}\label{rem:setup}
In Set-up \ref{setup-relcone}, since $q(X_s)=0$, the higher direct image $R^1 \pi_* \sO_X$ is a torsion sheaf on $S$. In fact, for a klt $K$-trivial fibration we can apply \cite[Theorem 1.2]{Men23} to obtain $R^1\pi_*\sO_X =0$. 
By \cite[Proposition 12.1.4]{KM92} this implies that $N^1(X/S)_{\Q}\simeq \pic(X/S)_{\Q}$.
Since $S$ is $\Q$-factorial, we know by \cite[Proposition 4.4(4)]{LZ22} that the relative movable cone $\Mov{X/S}$ is non-degenerate. This will allow  to apply Looijenga's results \cite{Loo14} (cf.\ Lemma \ref{lem:loo}) in the proof of Theorem \ref{thm:main} below. 
%
\end{remark}

As anticipated above we build upon the recent paper \cite{GLSW24} and give the relation between the relative movable and nef  cone conjectures.

\begin{theorem}\label{thm:main}
In the Set-up \ref{setup-relcone}, assume that 
\begin{equation}\label{eq:weak}
\textrm {good minimal models exist for effective klt pairs on the fibre $X_s$.}\footnote{See Terminology \ref{terminologygood} for a precise definition.}
\end{equation}
Then the following statements are equivalent.
\begin{enumerate}[(a)]
\item The action of $\PsAut(X/S)$ on $\Move{X/S}$ admits a rational polyhedral fundamental domain.
\item \begin{enumerate}
\item For every small $\Q$-factorial modification (SQM) $X'$ of $X$ over $S$, the action of $\Aut(X'/S)$ on $\Nefe{X'/S}$ admits a rational polyhedral fundamental domain.
\item Up to isomorphism over $S$, there are only finitely many SQMs of $X$ over $S$.
\end{enumerate}
\end{enumerate}
\end{theorem}

When $S$ is a point, this result has been proved in \cite[Theorem 1.5(iii)]{GLSW24}
assuming the existence of good minimal models in dimension $\dim(X)$, and in \cite[Theorem 13]{Xu24} assuming the existence of  minimal models and a non-vanishing conjecture in dimension $\dim(X)$.
Notice that the implication ``(a) implies (b)(ii)'' is  proved in \cite[Proposition 5.3(2)]{LZ22}, again under the stronger hypothesis of the existence of good minimal models for all effective klt pairs in dimension $\dim(X)-\dim(S)$. 
Building upon these  inspiring papers, a crucial part of our work consists in weakening their hypotheses to (\ref{eq:weak}), having in mind our application to families of IHS manifolds, for which minimal models exist by \cite{LP16}, and, as recalled above, abundance is proved for the 4 known deformation types.

From \cite[Theorem 1.4]{LZ22} we immediately deduce the following. 
\begin{corollary}\label{cor:K3s}
Let $\pi\colon X\to S$ a fibration in $K3$ surfaces such that $S$ is $\Q$-factorial. Then the action of $\Aut(X/S)$ on $\Nefe{X/S}$ admits a rational polyhedral fundamental domain.
\end{corollary}
This result also appeared in the very recent preprint \cite{MS24} more generally for any klt $K$-trivial fibration of relative dimension $2$, but with a weaker conclusion, see  \cite[Conjecture 1.1(1) and (2)]{MS24} for more details. 

To put this result into perspective we discuss a relevant example:
in Section \ref{sectionfamilyK3} we will 
construct a complete smooth family of K3 surfaces of Kummer type
$$
\holom{\pi}{X}{C}
$$
such that $\rho(X/C)=\rho(F)=17$ with $F$ a very general fibre $F$. 
It is well known that  any Kummer surface contains infinitely many smooth rational curves (see e.g.\ \cite[Example 5]{BHT11}), which, by Kov\'acs' theorem \cite[Theorem 1]{Kov94}, generate the Mori cone of $F$. In particular, the nef cone of $F$ is not polyhedral. 
Yet we can prove that for our family one has the following property.

\begin{proposition} \label{proposition:aut:finite}
The relative automorphism group $Aut(X/C)$ is finite. Thus, by 
Corollary \ref{cor:K3s}, it follows that the $\Nefe{X/C}$ is a rational polyhedral cone.
\end{proposition}

By way of contrast we will see in Subsection \ref{subsectionSQM}
that  the relative pseudoeffective cone
has infinitely many extremal rays and $\Bir(X/C)$ is infinite.

{\bf Acknowledgments.} We thank the authors of \cite{GLSW24} for sharing with us a preliminary version of their paper, and B. Cadorel, Y. Brunebarbe and Y. Deng for their help with Lemma \ref{lem:cadorel}, respectively Proposition \ref{propositionVsemistable}.  We thank G. Ancona, V. Lazi\'c and G. Mongardi for useful remarks on a preliminary version of the paper.  AH and GP are partially supported by the project ANR-23-CE40-0026 ``Positivity on K-trivial varieties''.
AH was supported by the France 2030 investment plan managed by the National Research Agency (ANR), as part of the Initiative of Excellence of Université Côte d’Azur under reference number ANR-15-IDEX-01.

\section{Notation and basic definitions}

We work over $\C$, for general definitions we refer to \cite{Har77}. 
We recommend Lazarsfeld's books for notions of positivity of $\R$-divisors, in particular
\cite[Section 2.2]{Laz04a} for pseudoeffectivity and the associated cones.
In this paper a fibration is a surjective projective morphism with connected fibres between normal quasi-projective varieties.

{\bf Notational convention:} the letter $\pi$ will always be used for structural morphism, e.g.\
$$
\holom{\pi}{X}{S},\ \holom{\pi_Y}{Y}{S},\ \holom{\pi_i}{Y_i}{S}, \ldots
$$

We use the terminology of \cite{KM98} for birational geometry. Note
however that we systematically work with boundary divisors with $\R$-coefficients, 
since this degree of generality is needed for Sections \ref{s:geo} and \ref{sec:equiv}.


\subsection{The relative setting}

Let $\pi\colon X\to S$ be a fibration between $\Q$-factorial normal varieties. 
We say that two divisors $D_1$ and $D_2$ on $X$ are $\pi$-linearly 
equivalent if there exists a divisor $B$ on $S$ such that $D_1 \sim D_2+\pi^*B$. In this case we will use the notation $D_1 \sim_\pi D_2$ or $D_1\sim_S D_2$. 

We denote by $\pic(X/S)$ the relative Picard group and for 
$\mathbb K = \Q$ or $\R$ we set
$$
\pic(X/S)_{\mathbb K}\coloneqq \pic(X/S)\otimes_{\Z} \mathbb K.
$$

\begin{definition}
Let $\holom{\pi}{X}{S}$ be a fibration between normal $\Q$-factorial varieties.
A $\Q$-divisor $D$ is
\begin{itemize}
\item[-] $\pi$-effective if $\pi_* \sO_X(mD)$ has positive rank for some $m \in \N$ sufficiently divisible.
\item[-] $\pi$-big if for some  sufficiently divisible $m \in \N$, the natural map
$$
{\textrm ev}:\pi^* \pi_* \sO_X(mD) \rightarrow \sO_X(mD)
$$
induces a rational map $X \dashrightarrow Y$ that is birational.
\item[-] $\pi$-movable if 
$$
 \textrm{codim}({\textrm Supp} \big( {\textrm coker }({\textrm ev}))\big)>1. 
$$
\end{itemize}
An $\R$-divisor $D$ is
$\pi$-effective (resp. $\pi$-big, resp.\ $\pi$-movable) if $D= \sum_{j} a_j D_j$ where the $D_j$ are 
  $\pi$-effective (resp.\ $\pi$-big, resp.\ $\pi$-movable) $\Q$-divisors and the
$a_j \in \R^+$.
\end{definition}

One checks easily that a $\Q$-divisor divisor is $\pi$-effective (resp.\ $\pi$-big) if its restriction to the generic fibre of $\pi$ is effective (resp.\ big). 

\begin{lemma}\label{lem:rel_to_abs}
Let $\holom{\pi}{X}{S}$ be a fibration between normal $\Q$-factorial varieties.
Let $D$ an $\R$-divisor on $X$. Then the following hold:
\begin{enumerate}[(a)]
\item If $D$ is $\pi$-effective, there exists $D' \sim_{\R, \pi} D$ that is effective.
\item If $D$ is $\pi$-ample, there exists $D' \sim_{\R, \pi} D$ that is ample.
\item If $D$ is $\pi$-big, there exists $D' \sim_{\R, \pi} D$ that is big.
\end{enumerate}
\end{lemma}

\begin{remark}\label{rem:relnefcase}
Note that the statement does not hold for $\pi$-nef or $\pi$-pseudoeffective classes, cf.\ \cite[Example 1.46]{KM98} for a counterexample.
\end{remark}

\begin{proof}[Proof of Lemma \ref{lem:rel_to_abs}]
Note that the first item immediately reduces to the $\Z$-divisor case:
if $D = \sum_j a_j D_j$, where the $D_j$ are $\pi$-effective divisors, then
$D_j  \sim_{\Q, \pi} D_j'$ with $D_j' \geq 0$ implies that
$$
D = \sum_j a_j D_j \sim_{\R, \pi} a_j D_j' \geq 0.
$$
Now assume that $D$ is a $\Z$-divisor.
Choose $m \in \N$ such that the direct image
$\pi_* \sO_X(mD)$ has rank at least one.  By asymptotic Riemann-Roch (see \cite[Example 1.2.19]{Laz04a} we know that for some sufficiently ample line bundle $H$ the sheaf
$\pi_* \sO_X(mD) \otimes H$ has a global section, so $D' \coloneqq D + \pi^* \frac{1}{m} H$ is $\Q$-effective.

The second item is \cite[Proposition 1.45]{KM98}. 

The third item follows from the other two and from the fact that $D \sim_{\Q, \pi} A + E$ with $A$ relatively ample and $E$ a $\pi$-effective divisor. Indeed, given a $\pi$-big divisor
$D$ and a relatively ample divisor $H$, we can find $\varepsilon>0$ such that
$(D-\varepsilon H)_{|F}$ is effective for a general fibre $F$. Thus $D-\varepsilon H$ is $\pi$-effective.
\end{proof}

\begin{definition}
Let $X$ be a normal $\Q$-factorial variety, and let $\holom{\pi}{X}{S}$ be a fibration. Two $\R$-divisors $D_1$ and  $D_2$ on $X$ are $\pi$-numerically equivalent
(notation $D_1 \equiv_\pi D_2$ or $D_1\equiv_S D_2$) if $D_1 \cdot C = D_2 \cdot C$ for every curve $C \subset X$ such that $\pi(C)$ is a point. 
\end{definition}

We denote by $N^1(X/S)$ the set of numerical equivalence classes and for 
$\mathbb K = \Q$ or $\R$ we set
$$
N^1(X/S)_{\mathbb K}\coloneqq N^1(X/S)\otimes_{\Z} \mathbb K.
$$
We have a natural surjection
$N^1(X) \rightarrow N^1(X/S)$
and a sequence
$$
0 \rightarrow \pi^* N^1(S) \rightarrow N^1(X) \rightarrow N^1(X/S) \rightarrow 0
$$
that is not necessarily exact in the middle (e.g.\ if $S$ is an elliptic curve and $X=S \times S$). 

We are now ready to define the cones inside $N^1(X/S)_\R$ we will be concerned with:
\begin{enumerate}
\item[-] $\Eff {X/S}$ is the cone generated by $\pi$-effective divisors, $\EffO {X/S}$ its interior and $\overline{\textrm{Eff}} (X/S)$ its closure. 
\item[-] $\textrm{Mov} (X/S)$ is the cone generated by $\pi$-movable divisors and $\Mov {X/S}$ its closure. We set 
$$
\Move {X/S}\coloneqq\Mov{X/S}\cap \Eff{X/S}
$$
and 
$$
\Movp{X/S}\coloneqq\textrm{conv} \big( \Mov{X/S}\cap N^1(X/S)_\Q \big).
$$ 
\item[-] $\Amp{X/S}$  is the cone generated by $\pi$-ample divisors and
$\Nef {X/S}$ its closure. We set 
$$
\Nefe{X/S}\coloneqq\Nef{X/S}\cap \Eff{X/S}
$$ 
and 
$$
\Nefp {X/S}\coloneqq\textrm{conv} \big( \Nef{X/S}\cap N^1(X/S)_\Q \big).
$$ 
\end{enumerate}

\begin{remark}\label{rem:interior}
    Let $\pi\colon X\to S$ be a fibration. Then $\Amp{X/S}$ is the interior of $\Nef{X/S}$, and $\Bigcone{X/S}$ is the interior of $\overline{\textrm{Eff}}(X/S)$. 
This is clear since given a class $[D]$ in the interior of the cone, $[D]-\varepsilon [H]$ with $[H]\in \Amp{X/S}$  (resp.\ with $[H]\in \Bigcone{X/S}$)
is still in the cone for some $\varepsilon>0$. Now we can use the results from the absolute case on the fibres.
\end{remark}

\subsection{Rational maps}\label{subsec:maps}

Let $\merom{\mu}{X}{Y}$ be a rational map between normal projective varieties and let
$\holom{\mu_0}{X_0}{Y}$ be the restriction to the locus $X_0 \subset X$, where $\mu$ is well-defined. 
The graph $\Gamma_\mu \subset X \times Y$ of $\mu$ is defined as the closure of the graph of $\mu_0$,
in particular it is a (not necessarily normal) projective variety.

Given such a rational map  $\merom{\mu}{X}{Y}$ between normal projective $\Q$-factorial varieties, denote by $\holom{p_X}{\Gamma_\mu}{X}$
$\holom{p_Y}{\Gamma_\mu}{Y}$ the projections. Then we define the pull-back and push-forward
$$
\mu_* \coloneqq \ClQ{X} \rightarrow \ClQ{Y}, \ D \ \mapsto  \ (p_Y)_* p_X^* D
$$
and pull-back
$$
\mu^* \coloneqq \ClQ{Y} \rightarrow \ClQ{X}, \ D \ \mapsto  \ (p_X)_* p_Y^* D.
$$
In general the maps are not well-behaved with respect to the composition of rational maps.

\begin{definition}
A birational map $\merom{\mu}{X}{Y}$ between normal projective varieties is a contraction if the inverse $\mu^{-1}$ does not contract a divisor.
\end{definition}

\begin{definition}
Let $\holom{\pi}{X}{S}$ and $\holom{\pi_Y}{Y}{S}$ be  fibrations. We say that a rational
map $\merom{\mu}{X}{Y}$ is a map over $S$ if this holds on the locus where $\mu$ is well-defined.
A map ${\mu}:{X/S}\dashrightarrow{Y/S}$ is a birational contraction over $S$ if it is a map over $S$ which is a contraction between $X$ and $Y$.
\end{definition}

\begin{remark} \label{remark-graph}
In the situation above, let $X_0 \subset X$ be the locus where $\mu$ is well-defined. Denote by $\widetilde{X \times_S Y}$ the unique irreducible component of $X \times_S Y$ that contains $X_0 \times_S Y$.
Then $\merom{\mu}{X}{Y}$ is a rational map over $S$ if and only if its graph 
$\Gamma_\mu \subset X \times Y$ is contained in $\widetilde{X \times_S Y}$: this is clear on $X_0$, the statement follows by taking the closure.  

Note that this implies that for a birational morphism over $S$, the projections $\holom{p_X}{\Gamma_\mu}{X}$
and $\holom{p_Y}{\Gamma_\mu}{Y}$ commute with the natural maps, i.e.\ we have
$$
\pi \circ p_X = \pi_Y \circ p_Y.
$$
\end{remark}

\begin{definition}\label{def:SQM}
Let $(X_i,B_i)$ with $i\in\{1,2\}$ be two klt pairs.
Let  $\mu\colon X_1\dashrightarrow X_2$ be a birational contraction over $S$ of  normal $\Q$-factorial projective varieties. We say that the birational map $\mu$ is a small $\Q$-factorial modification (SQM) of $(X_1,B_1)$ over $S$ if $\mu$ is an isomorphism in codimension one which is a morphism over $S$ such that $\mu_* B_1 = B_2$.
\end{definition}

\begin{definition} \label{def:psaut}
Let $\holom{\pi}{(X,B)}{S}$ be a fibration between normal quasi-projective varieties. A pseudoautomorphism of $(X,B)$
is a birational map 
$$
g \colon X \dashrightarrow X
$$
over $S$ that is an isomorphism in codimension one and $g_* B = B$. The set of pseudoautomorphisms
forms a group which we denote by $\PsAut(X/S, B)$.
\end{definition}

We now come a central definition of this paper:

\begin{definition} \label{definition-chambres}
Let ${\mu}\colon {X/S}\dashrightarrow{Y/S}$ be a birational contraction over $S$  between normal $\Q$-factorial projective varieties.
The relative Mori chamber for $\mu$ is defined as 
$$
\Eff{X/S, \mu} \coloneqq \mu^* \Nefe{Y/S} + \sum_{E \in\textrm{Exc}(\mu)}  \R^+ E.
$$
\end{definition}

\begin{remark*}
The relative Mori chamber $\Eff{X/S, \mu}$ is contained in the relative effective 
cone $\Eff{X/S}$: indeed by Lemma \ref{lem:rel_to_abs} any class $[D] \in  \Nefe{X/S}$
can be represented by an effective divisor. In view of the definition of $\mu$
it is clear that $\mu^* D$ is still effective. 
\end{remark*}

We prove a relative version of a regularity criterion due to Hu and Keel \cite[Lemma 1.7]{HK00} in the absolute case (see \cite[Proof of Lemma 4.3]{GLSW24} for more details). 

\begin{lemma} \label{lemma-regular}
Let ${\mu_i}\colon {X/S}\dashrightarrow{Y_i/S}$ with $i\in\{1,2\}$ be two birational contractions over $S$   between normal $\Q$-factorial projective varieties.
Suppose that there exist $\R$-divisors
\begin{itemize}
\item[-]$D_1$ on $Y_1$ that is $\pi_1$-ample
\item[-] $D_2$ on $Y_2$ that is $\pi_2$-nef
\end{itemize}
such that
$$
\mu_1^* D_1 + E_1 \equiv_\pi \mu_2^* D_2 + E_2,
$$
where the $E_i$ are effective $\mu_i$-exceptional $\R$-divisors. Then the composition
$\mu_1 \circ \mu_2^{-1}$ extends to a regular birational morphism $Y_2 \rightarrow Y_1$.
\end{lemma}

Note that the statement is not a direct consequence of \cite[Lemma 4.3]{GLSW24} since, as recalled in Remark \ref{rem:relnefcase}, in general 
it is not possible to replace $D_2$ by a $D_2' \equiv_{\pi_2} D_2$ such that $D_2'$ is nef (and not just relatively nef).
However we will see that the proof adapts without too much difficulty. We start with a preliminary lemma:

\begin{lemma} \label{lemma-diagram}
In the situation of Lemma \ref{lemma-regular}, let $\Gamma$ be the normalisation of the graph of the birational map
$\mu_1 \circ \mu_2^{-1}$, and let 
$$
\holom{p_i}{\Gamma}{Y_i}, \qquad \holom{p}{\Gamma}{X}
$$
be the projections. Then $p_2^* D_2$ is $p_1$-nef and $p_1^* D_1$ is $p_2$-nef.
\end{lemma}

\begin{proof} By symmetry it is sufficient to show the statement for $D_2$.
This is equivalent to showing that if $C \subset \Gamma$ is an irreducible curve such that
$p_1(C)$ is a point, then $p_2(C)$ is contained in a fibre of $\pi_2$.  Yet this is clear:
since $\mu_1$ and $\mu_2$ are birational maps over $S$,
so is the map $\mu_1 \circ \mu_2^{-1}$. By Remark \ref{remark-graph} we therefore
have $\pi_1 \circ p_1 = \pi_2 \circ p_2$. Thus if $p_1(C)$ is a point, then
so is $\pi_2\big(p_2(C)\big)$. In other words $p_2(C)$ is contained in a $\pi_2$-fibre.
\end{proof}

\begin{proof}[Proof of Lemma \ref{lemma-regular}]
Let $\Gamma$ be the normalisation of the graph of the birational map
$\mu_1 \circ \mu_2^{-1}$, and let 
$$
\holom{p_i}{\Gamma}{Y_i} \text{ for }i\in\{1,2\}, \qquad \holom{p}{\Gamma}{X}
$$
be the projections. By Lemma \ref{lemma-diagram} we know that $p_2^* D_2$ is $p_1$-nef and $p_1^* D_1$ is $p_2$-nef.
Now we can just copy the first part of the proof of \cite[Lemma 4.3]{GLSW24} to see that
$p_1^* D_1 \equiv_{\pi \circ p} p_2^* D_2$. Indeed, the key point (application of the negativity lemma) only uses the relative nefness established in Lemma \ref{lemma-diagram}.
By Lemma \ref{lem:rel_to_abs} we can assume, without loss of generality, that $D_1$ is ample. Let now $C$ be any curve contracted by $p_2$. Then 
$$
D_1 \cdot p_1(C) = p_1^* D_1 \cdot C = p_2^* D_2 \cdot C = 0
$$
shows that $C$ is contracted by $p_1$. Thus $p_1$ contracts the $p_2$-fibres, and the rigidity lemma \cite[Lemma 1.6]{KM98} yields a morphism $\tau\colon Y_2 \rightarrow Y_1$ such that
$\tau \circ p_2=p_1$. By construction this morphism extends the rational map
$\mu_1 \circ \mu_2$.
\end{proof}

\begin{definition}
Let ${\mu_i}\colon {X/S}\dashrightarrow{Y_i/S}$ with $i\in\{1,2\}$ be two birational contractions over $S$ between normal $\Q$-factorial projective varieties.
We say that the two contractions are isomorphic if there exists an isomorphism
\holom{\varphi}{Y_1}{Y_2} over $S$ such that $\mu_2=\varphi\circ \mu_1$.
\end{definition}

\begin{lemma}\label{lem:isom_MoriChamber}
Let ${\mu_i}\colon {X/S}\dashrightarrow{Y_i/S}$ with $i\in\{1,2\}$ be two birational contractions over $S$ between normal $\Q$-factorial projective varieties.
Then the following are equivalent:
\begin{enumerate}[(a)]
\item $\mu_1$ is isomorphic to $\mu_2$;
\item The interiors of the Mori chambers intersect, i.e.\ we have
$$
\EffO{X/S, \mu_1} \cap \EffO{X/S, \mu_2} \neq \emptyset.
$$ 
 \end{enumerate}
\end{lemma}

\begin{proof}
We proceed as in the proof of \cite[Lemma 4.2]{GLSW24}.   It is clear that $(a)$ implies $(b)$. For the other implication we observe that by Remark \ref{rem:interior} and \cite[Theorem 6.6]{Roc70} we have that the relative interior $\textrm{ri}\big(\mu_i^* \Nefe {Y_i/S}\big)$ coincides with $\mu_i^*(\Amp{Y_i/S})$. Notice moreover that 
 $
 \EffO{X/S, \mu_i}=\textrm{ri} \big( \Eff{X/S, \mu_i}\big)
 $, since the convex cone $\Eff{X/S, \mu_i}$ is full-dimensional in the relative effective cone $\Eff{X/S}$.
 Then from
 \cite[Corollary 6.6.2]{Roc70}, it follows that  
\begin{equation}\label{eq:b=>a}
 \EffO{X/S, \mu_i}=\textrm{ri} \big( \Eff{X/S, \mu_i}\big)= \mu_i^* \Amp{Y_i/S}+ \textrm{ri} (\sum_{E\in \textrm{Exc}(\mu_i)} \mathbb R_{\geq 0} E).
\end{equation}
Now from $(b)$ and (\ref{eq:b=>a}) we deduce the existence, for $i=1,2$, of relatively ample $\mathbb R$-divisors $A_i$ on $Y_i$, and $\mu_i$-exceptional divisors on $X$, such that 
$$
\mu_1^*A_1 + E_1= \mu_2^*A_2 + E_2. 
$$
By applying twice Lemma \ref{lemma-regular} we obtain that $\mu_1\circ \mu_2^{-1}\colon Y_1\to Y_2$ is biregular and yields the desired isomorphism. 
\end{proof}

\begin{remark}\label{rmk:=}
Notice that Lemma \ref{lem:isom_MoriChamber} implies in particular that as soon as $
\EffO{X/S, \mu_1} \cap \EffO{X/S, \mu_2} \neq \emptyset$, we actually have $
\EffO{X/S, \mu_1} = \EffO{X/S, \mu_2}.$
\end{remark}

%
\subsection{Cones and cone conjectures}\label{subsec:cones}
%

Let $V$ be a real vector space of finite dimension. We suppose that $V$ has a distinguished $\mathbb Q$-structure, i.e.\ there exists a rational vector space $V_{\mathbb Q}$ such that $V= V_{\mathbb Q}\otimes_{\mathbb Q}\mathbb R$.  A cone is a subset $C\subset V$ such that 
$$
\forall \lambda \in \mathbb R_{>0}, \forall x \in C, \mbox{ we have }  \lambda x\in C.
$$ 
A cone is called  polyhedral  (resp.\ rational polyhedral) if it is a closed
convex cone generated by a finite number of vectors (resp.\ rational vectors). 

\begin{definition}
If $C\subset V$ is an open convex cone, we denote by $C_+$  the convex hull in $V$ of $\overline {C}\cap V(\mathbb Q)$.

Let $\Gamma$ be a group and $\rho\colon \Gamma\to \textrm {GL}(V)$ a group homomorphism. Suppose that the action of  $\Gamma$ on $V$ (via $\rho$) leaves a convex full-dimensional cone $C\subset V$ invariant. We say that $(C_+,V)$ is of polyhedral type if there is a polyhedral cone $\Pi\subset C_+$ such that $\Gamma\cdot \Pi \supset C$. 
\end{definition}
The following important result is \cite[Theorem 3.8 and Application 4.14]{Loo14}, see also \cite[Lemma 3.5]{LZ22}. 
\begin{lemma}\label{lem:loo}
    Under the notation
and assumptions above,  suppose that $\rho\colon \Gamma\to \textrm {GL}(V)$ is injective and that the cone $C$ is non-degenerate. If $(C_+, \Gamma)$ is of polyhedral type, then  $C_+$  admits a rational polyhedral fundamental
domain for the action of $\Gamma$.
\end{lemma}


In the literature (and in Morrison's original formulation when $S$ is a point and without boundary divisor), the cone $\Nefe{X/S}$ (resp.\ $\Move{X/S}$) in Conjecture \ref{conj:cone}  is sometimes replaced with $\Nefp{X/S}$ (resp.\ $\Movp{X/S}$). 

The references directly related to our work are those establishing the conjecture for IHS manifolds: \cite{Mar11}  for the movable cone conjecture, \cite{AV17,AV20} for deriving the nef cone conjecture from the movable one (see also \cite{MY15}), \cite{LMP24} for singular IHS varieties, and \cite{Tak21} over not necessarily closed fields of characteristic zero, as well as recent and less recent works in the relative case  \cite{Kaw97, LZ22, Li23, MS24}. Moreover, as already mentioned, our paper is  inspired by \cite{GLSW24} (see also \cite{Xu24}). 

%
%
%
%
\subsection{Birational geometry}
%
We will follow the terminology from \cite[Section 2.3]{KM98}.
\begin{definition}
Let $X$ be a normal projective variety and let $\holom{\pi}{X}{S}$ be a fibration. Let $\Delta$ be a boundary divisor on $X$ such that the pair $(X,\Delta)$ is log-canonical. 

Let $\varphi\colon X\dashrightarrow Y$ be a birational contraction of normal projective varieties over $S$ and set $\Delta_Y \coloneqq \varphi_* \Delta$. We say that $\varphi\colon (X,\Delta)\dashrightarrow (Y,\Delta_Y)$ is a minimal model of $(X,\Delta)$ if
\begin{itemize}
    \item[-] $a(E,X,\Delta)> a(E,Y,\Delta_Y)$ for all $\varphi$-exceptional divisor,
    \item[-] $K_Y+\Delta_Y$ is nef over $S$, and
    \item[-] $Y$ is $\Q$-factorial.
\end{itemize}

Moreover, a minimal model $(Y,\Delta_Y)$ over $S$ is a good minimal model if $K_Y+\Delta_Y$ is semiample over $S$.
\end{definition}

Note that by \cite[Lemma 3.6.3]{BCHM10}, the first condition in the definition of a minimal model is equivalent to requiring that $\varphi$ is $(K_X+\Delta)$-negative.


\begin{terminology} \label{terminologygood}
Let $X$ be a normal projective $\Q$-factorial variety. We say that good minimal models
exist for effective klt pairs on $X$ if for every klt pair $(X, \Delta)$ such that
$\kappa(K_X +\Delta) \geq 0$, the pair $(X,\Delta)$ has a good minimal model.
\end{terminology}

\begin{remark}\label{rmk:gmm}
The existence of good minimal models for effective klt pairs is known in the following
cases (we limit ourselves to cases relevant for klt $K$-trivial fibrations): 
\begin{enumerate}[(a)]
\item {\it $\dim\leq 3$,} by the ``classical'' work of Mori, Kawamata and many others.  



\item {\it klt pairs fibered in varieties admitting a good minimal model,} by \cite{HX13} (see Theorem \ref{thm:HX13_R_divisor} below) building upon  \cite{Lai10}.

 \item {\it Projective IHS manifolds of the 4 known deformation types.} By \cite{LP16} for any effective divisor $B$ on any projective IHS manifolds $X$ such that $(X,B)$ is klt, the pair $(X,B)$ admits a minimal model.  Moreover the semiampleness of a square zero nef divisor (if the square is positive then the divisor is also big and one invokes the base-point-free theorem) is given by \cite[Corollary 1.1]{Mats17}
in the case of $K3^{[n]}$ or generalized Kummer deformation type, 
by \cite[Theorem 2.2]{MO22} in the case of OG10 deformation type, and by \cite[Theorem 7.2]{MR21} in the
case of OG6 deformation type.
\end{enumerate}
\end{remark}

The following result is \cite[Theorem 2.7]{LZ22}.
\begin{theorem}\label{thm:intersect_nef}
    Let $(X,B)\to S$ be a klt $K$-trivial fibration. Let $\Pi\subset \Eff{X/S}$ be a polyhedral cone.
    Then $\Pi\cap\Nefe{X/S}$ is a polyhedral cone. If moreover $\Pi$ is rational, then $\Pi\cap\Nefe{X/S}$ is also rational.
\end{theorem}

We also record the following version of \cite[Theorem 2.12]{HX13} from \cite[Theorem 2.2]{Li23} which holds more generally for $\R$-divisors. 

\begin{theorem}\label{thm:HX13_R_divisor}
Let $\pi\colon X\to S$ be a fibration and let $(X,B)$ be a klt pair. Assume that for a very general point $s\in S$, 
the pair $(X_s, B_s)$ has a good minimal model.
Then $(X,B)$ has a good minimal model over $S$.
\end{theorem}

\begin{proof}
Li observed in \cite[Theorem 2.2 and the following paragraph]{Li23} that the proof of  \cite[Theorem 2.12]{HX13} goes through by replacing $\textrm{Proj} \big(R(X/S, K_X+\Delta) \big)$ with  
the canonical model of $(X,\Delta)$ over $S$ whose existence is known for effective klt pairs by \cite[Corollary 1.2]{Li22}. 
%
\end{proof}

%
\section{The geography of models for Klt fibrations}\label{s:geo}
%

We now prove a result which is  stronger than \cite[Theorem 2.4]{LZ22}. Notice however that we follow closely the proof of \cite[Theorem 2.4]{LZ22}, checking that it still works with our weaker hypothesis. In its turn \cite[Theorem 2.4]{LZ22}  goes back to \cite[Theorem 3.4]{SC11} and \cite[Lemma 7.1]{BCHM10}. 

    \begin{theorem}\label{thm:geo+}
        Let $X$ be a normal projective $\Q$-factorial variety and $\pi\colon X\to S$ be a fibration. Assume that for a very general point $s\in S$, good minimal models exist for effective klt pairs on the fibre $X_s$.

        Fix $D_j$ effective $\Q$-divisors on $X$ with $j\in\{1,\dots,r\}$ and set $\mathcal{C}\coloneqq \bigoplus_{j=1}^r [0,1)D_j$.
        Suppose that $P\subset \mathcal{C}$ is a rational polytope such that
        \begin{equation}\label{eq:hyp_P}
            \text{for any }\Delta\in P, \text{ the pair }(X,\Delta)\text{ is klt.}
        \end{equation}
        Then there is a finite decomposition
        \[
        P = \bigsqcup_{i=1}^m P_i
        \]
        such that $\overline{P}_i$ is a rational polytope for any $1\leq i\leq m$, together with finitely many $\Q$-factorial birational contractions \[\varphi_i\colon X\dashrightarrow X_i \text{ over } S, \text{ where } 1\leq i\leq m,\] such that the following property holds:
        if $\Delta \in P_i$ for some $i\in\{1,\dots,m\}$, then $\varphi_i\colon (X,\Delta)\dashrightarrow (X_i,\Delta_i)$ with $\Delta_i\coloneqq (\varphi_i)_* \Delta$ is a good minimal model of $(X,\Delta)$ over $S$.
        \end{theorem}



      \begin{remark}\label{rem:conv_geo}
        To prove Theorem \ref{thm:geo+}, we start with the following remark on convex geometry. Let $P$ be a rational polytope.
    \begin{enumerate}[(a)]
        \item Given $\epsilon > 0$, the rational polytope $P$ can be decomposed into finitely many rational sub-polytopes $P_i$'s such that for any $i$, we have $\dim P_i = \dim P$ and that $P_i$ is contained in a ball of radius $\epsilon$ in the affine hull of $P$. Moreover, if we denote $\textrm{ri}(\cdot)$ the relative interior, for any $i\neq j$ we have $\textrm{ri}(P_i)\cap \textrm{ri}(P_j)=\emptyset$.
        \item Fix a decomposition $P=\bigcup_i P_i$ into finitely many rational sub-polytopes $P_i$'s with $\dim P_i = \dim P$ for any $i$, and $\textrm{ri}(P_i)\cap\textrm{ri}(P_j)=\emptyset$ for any $i\neq j$.
        
        Assume that for each $P_i$, there exists a finite disjoint
        decomposition $P_i=\bigsqcup_\ell Q_{i,\ell}$ such that
        $\overline{Q}_{i,\ell}$ is a rational polytope for any
        $\ell$. Then we can construct a finite disjoint decomposition
        $P=\bigsqcup_k Q_k$ which is compatible with $P=\bigcup_i P_i$
        and such that $\overline{Q}_k$ is a rational polytope for any
        $k$. To prove the assertion, we fix the following terminology (see \cite[Section 3]{SC11}): let $P'\subset P$ be a polytope, we denote by $\textrm{Int}_P(P')$ the biggest open convex polyhedron in $P$ which is contained in $P'$; by open convex polyhedron in $P$, we mean a set defined by the intersection of $P$ with finitely many hyperplanes and open halfspaces in the linear span of $P$.
        Indeed, to construct the components $Q_k$'s, we first consider $\textrm{Int}_P (P_i)$ for all $i$ and we take all the components in its induced disjoint decomposition as components $Q_k$'s. We then continue with the faces of the $P_i$'s for all $i$, which do not lie on the boundary of $P$: we consider these faces by decreasing dimension order, we take $\textrm{Int}_P (\cdot)$ of each face and we add it successively as components $Q_k$'s.
    \end{enumerate}
    \end{remark}

         \begin{proof}
        We first note that by assumption and Theorem \ref{thm:HX13_R_divisor},
        \begin{equation}\label{eq:every_point_gmm}
            \text{ for any }\Delta\in P, \text{ the klt pair }(X,\Delta) \text{ has a good minimal model over }S.
        \end{equation}

        We proceed by induction on the dimension of $P$. The case $\dim P =0$ is clear by \eqref{eq:every_point_gmm}. We consider the case $\dim P \geq 1$.

    By Remark \ref{rem:conv_geo},
    it suffices to prove the statement locally around any point $\Delta\in P$, which means to prove that for any point $\Delta\in P$, there exists a rational polytope $P'\subset P$ containing $\Delta$ such that $\dim P'=\dim P$ and that $P'$ satisfies the decomposition properties in the statement of the theorem. Moreover, we can always shrink the rational polytope $P$, which means that we replace $P$ by a smaller rational polytope $P'\subset P$ such that $\Delta\in P'$ and $\dim P'=\dim P$. 
    
    \medskip

    \emph{Step 1. We prove the claim under the following assumption:
    there exists a rational point $\Delta_0\in \textrm{ri}(P)$ such that $K_X+\Delta_0\equiv_S 0$.}

    Since $\Delta_0\in\textrm{ri}(P)$ and $\dim P\geq 1$, we can pick a point $\Delta\in P$ such that $\Delta\neq \Delta_0$.
    Let $\Delta_{\mathcal{B}}$ be the projection of $\Delta$ from $\Delta_0$ on the boundary of $P$.

    Let $Q$ be a face of $P$ such that $\dim Q = \dim P -1$ and $\Delta_{\mathcal{B}}\in Q$. Let \[P' \coloneqq \textrm{conv}(Q,\Delta_0).\] 
    Then $P'$ is a rational polytope such that $\dim P'= \dim P$ and $\Delta\in P'\subset P$. By the inductive hypothesis, we can decompose $Q$ into a finite disjoint union
    \[
    Q =  \bigsqcup_{k=1}^s Q_k, 
    \]
    such that $\overline{Q}_k$ is a rational polytope for any $1\leq k\leq s$,
    together with finitely many $\Q$-factorial birational contractions 
    \[\theta_k\colon X\dashrightarrow X'_k \text{ over } S \text{ with } k\in\{1,\dots,s\}\] 
    satisfying the following property:
        if $\Gamma \in Q_k$ for some $k\in\{1,\dots,s\}$, then $\theta_k\colon (X,\Gamma)\dashrightarrow (X'_k,\Gamma'_k)$ with $\Gamma'_k\coloneqq (\theta_k)_* \Gamma$ is a good minimal model of $(X,\Gamma)$ over $S$.
    We note that any point in $P'\backslash \{\Delta_0\}$ can be written as $t\Gamma + (1-t)\Delta_0$ for some $t\in (0,1]$ and some $\Gamma\in Q$. Since
        \[
        K_X + t\Gamma + (1-t)\Delta_0 = t(K_X+\Gamma) + (1-t)(K_X+\Delta_0) \equiv_S t(K_X+\Gamma) \text{ for any }t\in(0,1],
        \]
    setting $\Delta'_{0,k}\coloneqq  (\theta_k)_* \Delta_0$, we infer that $(X_k, t\Gamma'_k + (1-t)\Delta'_{0,k})$ is a good minimal model of $(X, t\Gamma + (1-t)\Delta_0)$ over $S$ if and only if $(X_k, \Gamma'_k)$ is a good minimal model of $(X, \Gamma)$ over $S$.

    For any $k\in\{1,\dots,s\}$, we denote \[P_k \coloneqq \textrm{conv}(Q_k,\Delta_0)\backslash\{\Delta_0\}.\]
    Hence, by the above
        \[
        P' = \{ \Delta_0 \} \bigsqcup \Bigl( \big(\bigsqcup_{k=1}^s P_k\big) \Bigr)
        \]
    is a decomposition satisfying the desired decomposition properties; in particular, for each $k$, the $\Q$-factorial birational map is given by $\theta_k$. 

    \medskip

    \emph{Step 2. In this step we proceed with the general case.}

    Since $P$ is rational, there exists always a rational point in $\textrm{ri}(P)$.
    Fix a rational point $\Delta_0\in \textrm{ri}(P)$. Then $(X,\Delta_0)$ has a good minimal model over $S$ by \eqref{eq:every_point_gmm}, yielding a birational map $\eta\colon (X,\Delta_0)\dashrightarrow (X',\Delta_0')$ over $S$, where $\Delta_0'\coloneqq \eta_* \Delta_0$ is semiample over $S$. Hence, there exists a contraction 
    \[
    \phi\colon X' \to Z' \text{ over }S
    \]
    and $\Q$-Cartier divisor $A'\subset Z'$ which is ample over $S$ such that
    \begin{equation}\label{eq:pullbak_of_A'}
    K_{X'} + \Delta_0' \sim_{\Q,S} \phi^* A'.
    \end{equation}
    In particular, we have
    \begin{equation}\label{eq:trivial_over_Z'}
     K_{X'} + \Delta_0' \equiv_{Z'} 0.
    \end{equation}
    Since $\eta$ is $(K_X+\Delta_0)$-negative, by shrinking $P$, we may assume that
    \begin{equation}\label{eq:eta_is_negative}
        \eta \text{ is }(K_X+\Delta)\text{-negative for any }\Delta\in P,
    \end{equation}
    which implies that $P_{X'}\coloneqq \eta_*(P)$ satisfies the assumption \eqref{eq:hyp_P}, namely, for each divisor $\Delta'\in P_{X'}$ the pair $(X',\Delta')$ is klt.

By \eqref{eq:trivial_over_Z'} and \emph{Step 1}, the polytope $P_{X'}$ can be decomposed into a finite disjoint union
    \[
    P_{X'} = \bigsqcup_{i=1}^m Q'_i
    \]
    such that $\overline{Q'}_i$ is a rational polytope for any $i\in\{1,\dots,m\}$,
    and there exist finitely many $\Q$-factorial birational maps $\varphi_i'
    \colon X'\dashrightarrow X''_i$ over $Z'$
    such that for any $\Delta'_i\in Q_i'$, if we set $\Delta''_i\coloneqq (\varphi_i')_*\Delta'_i$, then
    \begin{equation}\label{eq:induction_good_mim_model}
    \varphi_i'\colon (X',\Delta'_i)\dashrightarrow (X''_i,\Delta''_i) \text{ is a good minimal model of }(X',\Delta'_i)\text{ over }Z'.
    \end{equation}

    \medskip

    \emph{Claim}. After shrinking $P$,
    for any $\Delta'_i \in Q_i'$,
    \[
    K_{X''_i} + \Delta''_i \text{ is nef over }S.
    \]

    The claim implies that for any $\Delta'_i \in Q_i'$,
    \[\varphi_i'\colon (X',\Delta'_i)\dashrightarrow (X''_i,\Delta''_i) \text{ is a minimal model of } (X',\Delta'_i) \text{ over }S.\]
   
   {\it Proof of the Claim.}
    Let $\Gamma'_i$ be a vertex of $\overline{Q'}_{i}$.
    Let $\Gamma_i\in P$ such that $\eta_*\Gamma_i = \Gamma'_i$.
    Then $(X,\Gamma_i)$ has a good minimal model over $S$ by \eqref{eq:every_point_gmm}. By \eqref{eq:eta_is_negative} and since $\varphi'_i$ is a good minimal model of $(X',\Gamma'_i)$ over $Z'$, $(X''_i,\Gamma''_i)$ is a good minimal model of $(X,\Gamma_i)$ over $Z'$. In particular, $(X''_i,\Gamma''_i)$ is semiample over $Z'$.
    Therefore, there exists a contraction
    \[
    \tau_i\colon X''_i\to T_i \text{ over }Z'
    \]
    and a $\Q$-Cartier divisor $H_i\subset T_i$ which is ample over $Z'$ such that
    \[
    K_{X''_i} + \Gamma''_i \sim_{\Q,Z'} \tau_i^* H_i.
    \]
    Thus there exists a $\Q$-Cartier divisor $\Theta'\subset Z'$ such that
    \begin{equation}\label{eq:pullback_on_X''_i}
    K_{X''_i} + \Gamma''_i = \tau_i^* H_i + (\mu_i\circ\tau_i)^* \Theta',
    \end{equation}
    where $\mu_i\colon T_i\to Z'$.
    \begin{center}
       \begin{tikzcd}
X \arrow[r, "\eta", dashed] \arrow[d, "\pi"'] & X' \arrow[r, "\varphi'_i", dashed] \arrow[d, "\phi"] & X''_i \arrow[d, "\tau_i"] \\
S                                             & Z' \arrow[l]                                & T_i \arrow[l, "\mu_i"']  
\end{tikzcd}
    \end{center}
    Then there exists some rational number $0<t_0\ll 1$ such that
    \[
    t(H_i + \mu_i^* \Theta') + (1-t) \mu_i^* A' = tH_i+ \mu_i^*\big( (1-t)A'+t\Theta'\big) \text{ is nef over }S
    \]
    for any $t\in[0,t_0]$, where $A'$ is as in (\ref{eq:pullbak_of_A'}), because the divisor $(1-t)A'+t\Theta'$ remains ample over $S$ for small $t$.

    We set $\Delta''_{0,i}\coloneqq (\varphi'_i)_* \Delta'_0$. Hence, by \eqref{eq:pullback_on_X''_i} and \eqref{eq:pullbak_of_A'}, we obtain
    \[
    t( K_{X''_i} + \Gamma''_i) + (1-t)(K_{X''_i} + \Delta''_{0,i}) \sim_{\R,S} \tau_i^*\big(t(H_i + \mu_i^* \Theta') + (1-t) \mu_i^* A'\big).
    \]
    Replacing $\Gamma'_i$ by $t_0\Gamma'_i + (1-t_0)\Delta'_0$ and repeating the above process for each vertex of $Q'_i$, we get a closed set $P'$ with $\Delta'_0\in P'\subset P$ and $\dim P'=\dim P$, such that $P'$ satisfies the claim. Since there are finitely many $Q'_i$, we can shrink $P'$ to obtain a rational polytope satifying the claim.
    
    \medskip

    \emph{Step 3}. By \emph{Step 2} we have a decomposition of $P_{X'}$. We now show how to deduce a decomposition of $P$ satisfying the desired decomposition properties. 
    
    By \eqref{eq:eta_is_negative} and by \emph{Step 2},  
     the decomposition
    \[
    P_{X'} = \bigsqcup_{i=1}^m Q'_i
    \]
    induces, via the map $\eta_*\colon N^1(X/S) \to N^1(X'/S)$, a decomposition  $P= \bigsqcup_{i=1}^m Q_i$  of $P$. More explicitly,  for any $\Delta\in Q_i$, if we denote $\Delta''_i\coloneqq (\eta\circ\varphi'_i)_* \Delta$, we obtain that $\eta\circ \varphi'_i\colon (X,\Delta)\dashrightarrow (X''_i,\Delta''_i)$ is a minimal model of $(X,\Delta)$ over $S$. 
    
    Since $(X,\Delta)$ has a good minimal model over $S$ by \eqref{eq:every_point_gmm}, we infer by \cite[Lemma 2.4(3)]{HX13} that $(\eta\circ \varphi'_i)\colon (X,\Delta)\dashrightarrow (X''_i,\Delta''_i)$ is also a good minimal model of $(X,\Delta)$ over $S$. Thus the decomposition $P= \bigsqcup_{i=1}^m Q_i$ satisfies the desired properties.
    \end{proof}

The following theorem is (a slightly stronger version of) \cite[Theorem 2.6]{LZ22}, which in turn heavily relies on the geography of models due to Choi and Shokurov \cite[Theorem 3.4]{SC11}. The hypothesis {\it existence of good minimal models in dimension $\dim(X)-\dim(S)$} can be replaced by {\it existence good minimal models
 for effective klt pairs on the fibre $X_s$} as in the proof one can invoke Theorem \ref{thm:geo+} instead of \cite[Theorem 2.4]{LZ22}.

 \begin{theorem} \label{thm:ChoiShokurov}
    Let $(X,B)\to S$ be a klt $K$-trivial fibration. 
 Assume that for a very general point $s\in S$, 
good minimal models
exist for effective klt pairs on the fibre $X_s$.         

Then for any rational polyhedral cone $\Pi\subset \Eff{X/S}$, there is a finite decomposition
\[
\Pi = \bigcup_{i=1}^m \Pi_i
\] 
such that $\overline{\Pi}_i$ is a rational polyhedral cone for any $1\leq i \leq m$, together with finitely many $\Q$-factorial birational contractions
$$
\varphi_i\colon X\dashrightarrow X_i \text{ over }S, \text{ where }1\leq i\leq m,
$$
such that the following property holds: for any $\Delta\in \Pi_i$ such that  $(X,B+\epsilon \Delta)$ is  klt for some $\epsilon >0$,
the birational map
$$
\varphi_i\colon (X, B+\epsilon \Delta) \dashrightarrow (X_i, B_i+\epsilon \Delta_i)
$$
is a good model over $S$, where $B_i\coloneqq (\varphi_i)_* B$, and $\Delta_i\coloneqq (\varphi_i)_* \Delta$.    
\end{theorem}    




In other words we can  decompose any  rational polyhedral cone $\Pi\subset \Eff{X/S}$ into finitely many  sub-cones so that for any of these there exists a $\Q$-factorial birational contraction $\varphi_i\colon X\dashrightarrow X_i$ over $S$ that is a good minimal model for any divisor class in $\textrm{ri}(\Pi_i)$. 
Notice that the good minimal models $\varphi_i$ do exist by Theorem \ref{thm:HX13_R_divisor} and our assumption on the very general fibre.

We record the following result. 
\begin{lemma}\label{lem:otherinc}
Let $(X,B)\to S$ be a klt $K$-trivial fibration. 
 Assume that for a very general point $s\in S$, 
good minimal models
exist for effective klt pairs on the fibre $X_s$. Then     
$\Movp{X/S}\supset \Move{X/S}$.
\end{lemma}
\begin{proof}
This statement corresponds to \cite[Lemma 5.1(2)]{LZ22}. The proof given there works verbatim under our weaker hypothesis once we replace \cite[Theorem 2.6]{LZ22} with Theorem \ref{thm:ChoiShokurov}.
\end{proof}
%
\section{On  the relative movable and nef cone conjectures}\label{sec:equiv}
%

\begin{lemma}\label{lem:decomp_indices}
    With the notations and assumptions from Theorem \ref{thm:ChoiShokurov}, consider a $\Q$-factorial birational contraction $\mu\colon X\dashrightarrow Y$    over $S$, and set 
    $$
    I(\mu)\coloneqq \{ 1\leq i\leq m \mid \Pi_i\cap \EffO{X/S;\mu} \neq \emptyset \}.
    $$ Then the following hold:
   \begin{equation}\label{eq:indexset}
        I(\mu)= \{1\leq i\leq m \mid (X_i,\varphi_i)\simeq (Y,\mu)\},
    \end{equation}  
    \begin{equation}\label{eq:inclusions}
    \Pi\cap \EffO{X/S,\mu} \subset \bigcup_{i\in I(\mu)} \overline{\Pi}_i \subset \Pi \cap \overline{\Eff{X/S,\mu}}.
    \end{equation}   
    If moreover $\Pi\cap \EffO{X/S,\mu}\neq\emptyset$, then
    \begin{equation}\label{eq:decomp_intersect_MoriChamber}
        \bigcup_{i\in I(\mu)} \overline{\Pi}_i = \Pi \cap \overline{\Eff{X/S,\mu}}.
    \end{equation}
\end{lemma}

\begin{remark*}
During the proof of the statement we will show that
    \begin{equation}\label{eq:equiv_MoriChamber}
        i\in I(\mu) \text{ if and only if } \Eff{X/S,\mu} = \Eff{X/S,\varphi_i}.
    \end{equation}
\end{remark*}

\begin{proof}
    We first prove \eqref{eq:indexset}. 
    Let $i\in\{1,\dots,m\}$.
    By Theorem \ref{thm:ChoiShokurov}, we infer that $\Pi_i$ is contained in the Mori chamber $\Eff{X/S,\varphi_i}$, and thus 
    \[\Pi_i\subset \overline{\Pi}_i \subset \varphi_i^*\Nef{X/S}+ \sum_{E \in \textrm{Exc}(\varphi_i)}\R^+ E\]
    as the cone on the right is closed. 
    Moreover, we have $\Pi_i\subset \Pi \subset \Eff{X/S}$ by construction. Hence, we conclude that
    \begin{equation}\label{eq:Pi_i_in_MoriChamber}
        \Pi_i \subset \Eff{X/S,\varphi_i} \text{ for any }1\leq i \leq m.
    \end{equation}
    Therefore, if $i\in I(\mu)$, then $\EffO{X/S,\mu}\cap \Eff{X/S,\varphi_i}\neq \emptyset$. Hence by \cite[Corollary 6.3.2]{Roc70}, we obtain
    \begin{equation}\label{eq:interstmorichamber}
    \EffO{X/S,\mu}\cap \EffO{X/S,\varphi_i}\neq \emptyset.
    \end{equation}
    By Lemma \ref{lem:isom_MoriChamber}, \eqref{eq:interstmorichamber} holds if and only if the $\Q$-factorial birational contractions $\mu \colon X\dashrightarrow Y$ and $\varphi_i \colon X\dashrightarrow X_i$ over $S$ are isomorphic. This proves \eqref{eq:indexset}. In particular, see Remark \ref{rmk:=}, this establishes \eqref{eq:equiv_MoriChamber}.

    Now we prove the inclusions \eqref{eq:inclusions}. 
    The first inclusion follows from the fact that for any $1\leq j\leq m$, we have:
    \[
    \Pi_j\cap \EffO{X/S,\mu} \subset \bigcup_{i\in I(\mu)}\overline{\Pi}_i.
    \]
    The second inclusion follows from the relation
    \[
    \Pi_i \subset \Eff{X/S,\varphi_i} = \Eff{X/S,\mu} \text{ for any }i\in I(\mu)
    \]
    by \eqref{eq:Pi_i_in_MoriChamber} and \eqref{eq:equiv_MoriChamber}.

    Finally, assume $\Pi\cap \EffO{X/S,\mu}\neq\emptyset$. Then by \cite[Corollary 6.3.2]{Roc70}, we have
    \[
    \Pi^\circ \cap \EffO{X/S,\mu}\neq\emptyset.
    \]
    Hence, by \cite[Theorem 6.5]{Roc70}, we obtain
    \[
    \overline{\Pi\cap\EffO{X/S,\mu}} = \Pi \cap \overline{\Eff{X/S,\mu}}.
    \]
    Together with the inclusions \eqref{eq:inclusions} and since $\bigcup_{i\in I(\mu)} \overline{\Pi}_i$ is closed, this yields \eqref{eq:decomp_intersect_MoriChamber}.
    \end{proof}

\begin{proposition}\label{prop:finite_decomp}
    Let $\pi\colon (X, B) \to S$ be a klt $K$-trivial fibration.
     Assume that for a very general point $s\in S$, 
good minimal models
exist for effective klt pairs on the fibre $X_s$.   
    
Let $\Pi\subset \Eff{X/S}$ be a polyhedral cone. Then there is a finite chamber decomposition
    \[
    \Pi = \bigcup_{i=1}^m \Pi\cap \overline{\Eff{X/S,\varphi_i}},
    \]
    where every $\varphi_i\colon X\dashrightarrow X_i$ is an isomorphism class of a $\Q$-factorial birational contraction over $S$ such that $\EffO{X/S,\varphi_i}\cap \Pi \neq \emptyset$ and every closed chamber $\Pi\cap \overline{\Eff{X/S,\varphi_i}}$ is polyhedral.
    
    Moreover, if $\Pi$ is rational, then every closed chamber $\Pi\cap \overline{\Eff{X/S,\varphi_i}}$ is also rational.
\end{proposition}

    \begin{proof}
    We start by the following claim.

    \medskip
    
    {\em Claim.} We may assume that $\Pi\subset\Eff{X/S}$ is a rational polyhedral cone.

     {\it Proof of the Claim.}  We first notice that  from Lemma \ref{lem:rel_to_abs}(a) we deduce 
        $$\Eff{X/S}\subset \Eff{X/S}^+.$$ 
  Indeed, if $E$ is an $\mathbb \R$-divisor which is $\pi$-effective, by Lemma \ref{lem:rel_to_abs}(a) we may actually assume that $E$ is effective, i.e.\ 
  $E=\sum_k a_k E_k$ with $E_k$ effective Cartier divisor and $a_k\in \mathbb Q_{>0}$ for any $k$. As $E_k\in \Eff{X}^+\subset \Eff{X/S}^+$, we are done.  
        Hence, by enlarging $\Pi$ we may assume that $\Pi$ is a rational polyhedral cone. 
        
        To prove the claim, it suffices to show that if the statement of the proposition holds for a polyhedral cone $\Pi'\subset \Eff{X/S}$, then it also holds for any polyhedral subcone $\Pi\subset \Pi'$. Indeed, the finiteness of the decomposition for $\Pi$ follows from the finiteness of the decomposition for $\Pi'$. The polyhedrality of the closed chambers in the decomposition for $\Pi$ follows from the fact that
        \[
        \Pi \cap \overline{\Eff{X/S,\varphi_i}} = \Pi\cap \big( \Pi'\cap \overline{\Eff{X/S,\varphi_i}} \big)
        \]
        and the intersection of two polyhedral cones is a polyhedral cone. This proves the claim.

        \medskip
    
        Now we can apply Theorem \ref{thm:ChoiShokurov} to obtain a finite decomposition
        \[
        \Pi = \bigcup_{i=1}^m \Pi_i
        \]
        such that for each $1\leq i\leq m$, the closure $\overline{\Pi}_i$ is a rational polyhedral cone and there exists a $\Q$-factorial birational contraction $\varphi_i\colon X\dashrightarrow X_i$ over $S$ that is a good minimal model for any divisor class in $\Pi_i$.

        With the notation from Lemma \ref{lem:decomp_indices}, by \eqref{eq:equiv_MoriChamber} we infer that two relative Mori chambers $\Eff{X/S,\mu}$ and $\Eff{X/S,\mu'}$ of $\Q$-factorial birational contractions $\mu$ and $\mu'$  are distinct if and only if the two index sets $I(\mu),I(\mu')\subset\{1,\dots,m\}$ are disjoint. Hence, by \eqref{eq:decomp_intersect_MoriChamber} there are finitely many relative Mori chambers $\Eff{X/S,\mu}$ of $\Q$-factorial birational contractions $\mu$ such that $\EffO{X/S,\mu}\cap \Pi\neq \emptyset$. 
        Thus we obtain the finiteness of the decomposition for $\Pi$, i.e.
        \[
    \Pi = \bigcup_{\substack{1\leq i\leq m,\\ \EffO{X/S,\varphi_i}\cap \Pi\neq \emptyset} }\Pi\cap \overline{\Eff{X/S,\varphi_i}}.
    \]

    For the polyhedrality of the closed chambers in the decomposition for $\Pi$, take a relative Mori chamber  of a $\Q$-factorial birational contraction $\Eff{X/S,\mu}$ such that $\EffO{X/S,\mu}\cap \Pi\neq \emptyset$. By \eqref{eq:decomp_intersect_MoriChamber} we have
    \[
    \Pi \cap \overline{\Eff{X/S,\mu}}=\bigcup_{i\in I(\mu)} \overline{\Pi}_i.
    \]
    In the above equality, since the right hand side is a finite union of rational polyhedral cone and the intersection on the left hand side is a convex cone, we infer that $\Pi \cap \overline{\Eff{X/S,\mu}}$ is a rational polyhedral cone, as desired.
    \end{proof}

    \begin{proposition}\label{prop:mov_implies_nef}
    Let $\pi\colon (X, B) \to S$ be a klt $K$-trivial fibration as in the Set-up \ref{setup-relcone}.
     Assume that for a very general point $s\in S$, 
good minimal models
exist for effective klt pairs on the fibre $X_s$.       
    
If there is a polyhedral cone $\Pi\subset \Eff{X/S}$ satisfying
        \begin{equation}\label{eq:presquebir}
        \MovO{X/S}\subset \PsAut(X/S, B) \cdot \Pi,
        \end{equation}
        then the following statements hold.
        \begin{enumerate}[(a)]
        \item Up to isomorphism over $S$, there exist only finitely many $X'$ arising as SQM of $X$ over $S$.
        \item Let $X \dashrightarrow X'$ be a SQM over $S$. Then the 
        relative nef cone conjecture holds for $X'/S$.
        \end{enumerate}
    \end{proposition}

    \begin{proof}
        Let $\Pi\subset \Eff{X/S}$ be a polyhedral cone satisfying \eqref{eq:presquebir}.
        By Proposition \ref{prop:finite_decomp}
there is a finite chamber decomposition
    \[
    \Pi = \bigcup_{i=1}^m \Pi\cap \overline{\Eff{X/S,\varphi_i}},
    \]        
where the $\varphi_i$ are birational contractions over $S$.        
In particular, after renumbering,  there are finitely many SQM $\varphi_i\colon X\dashrightarrow X_i$ over $S$, say $1\leq i\leq m' \leq m$, such that \[\Pi\cap \EffO{X/S,\varphi_i}\neq \emptyset \text{ for any }1\leq i\leq m'.\]

        We first show (a).         
         Let $X'\to S$ be a klt $K$-trivial fibration arising as an SQM $\alpha \colon X\dashrightarrow X'$ over $S$. Since $\alpha$ is small the set $\Exc(\alpha)$ contains no divisors, hence we have $\Eff{X/S;\alpha} = \alpha^* \Nefe{X'/S}$. Hence, by Remark \ref{rem:interior}, we obtain
        $\EffO{X/S,\alpha} = \alpha^*\Amp{X'/S}$.
        
        By \eqref{eq:presquebir}, there exists an element $g\in \PsAut(X/S)$ such that $\Pi\cap \EffO{X/S,\alpha\circ g}\neq \emptyset$, as
        \[
        \EffO{X/S,\alpha\circ g} = g^* \EffO {X/S,\alpha} = g^* \alpha^*\Amp{X/S}\subset \MovO{X/S}.
        \]
        Therefore, by Lemma \ref{lem:decomp_indices}, item (\ref{eq:decomp_intersect_MoriChamber}) 
        there exists $i\in\{1,\dots,m\}$ such that $\EffO{X/S,\alpha\circ g} \cap \Eff{X/S,\varphi_i}\not= \emptyset$. Thus by \cite[Corollary 6.3.2]{Roc70} and Lemma \ref{lem:isom_MoriChamber}, we have $\Eff{X/S,\alpha\circ g} = \Eff{X/S,\varphi_i}$, and $X'$ and $X_i$ are isomorphic over $S$, which yields (a).

        Now we show (b). Let $(X',B')$ be a representative of one of the finitely many isomorphism classes of SQM of $(X,B)$ over $S$. We observe that the existence of
        good minimal models for effective klt pairs on the fibre $X_s$ easily implies the same for the       
        very general fibre $X'_s$ of $X'\to S$. Hence we set $X=X'$ in what follows. 
        
        Let $J$ be a finite index subset of $\{1,\dots,m\}$ such that for any $j\in J$ there exists $g_j\in\PsAut(X/S)$ satisfying $\Eff{X/S,g_j}=\Eff{X/S,\alpha_j}$, where the $\alpha_j:X\dashrightarrow X_j$ are the finitely many SQMs of $X$ over $S$ (up to isomorphism) given by (a). Note that the $g_j$'s exist by hypothesis (\ref{eq:presquebir}), Lemma  \ref{lem:decomp_indices} and Lemma \ref{lem:isom_MoriChamber}. By Theorem \ref{thm:intersect_nef}, for any $j\in J$, the intersection
        \[
        (g_j^{-1})^* \Pi \cap \Nef{X/S} \subset \Eff{X/S}
        \]
        is a polyhedral cone. Let
        \[\Sigma \coloneqq \textrm{conv}\{  (g_j^{-1})^* \Pi \cap \Nef{X/S}\mid j\in J\}.\]
        Since $J$ is finite, the cone $\Sigma$ is also polyhedral and contained in $\Nefe{X/S}$.

         It suffices to show that
        \[
        \Amp{X/S}\subset \Aut(X/S)\cdot \Sigma.
        \]
        Indeed  by \cite[Lemma 5.1(1)]{LZ22} we have $\Nefe{X/S}\subset\Nefp{X/S}$, we can apply Looijenga's result Lemma \ref{lem:loo} to deduce the  relative  cone conjecture for $\Nefe{X/S}$.
        
        Let $D\in \Amp{X/S}$. By \eqref{eq:presquebir} there exists an element $g\in \PsAut(X/S)$ such that $g^* D \in \Pi$. In particular, we have $g^*\Amp{X/S}\cap \Pi \neq \emptyset$. Since $g$ is small, we infer that
        \[ \EffO{X/S,g} \cap \Pi = g^*\Amp{X/S} \cap \Pi \neq \emptyset.\]
        Therefore, by Lemma \ref{lem:decomp_indices} (see also \eqref{eq:equiv_MoriChamber}), there exists $j\in J$ such that $\EffO{X/S,g} \cap\Eff{X/S,g_j}\not=\emptyset$. Hence, by \cite[Corollary 6.3.2]{Roc70} we have $\EffO{X/S,g} \cap\EffO{X/S,g_j}\not=\emptyset$. Thus $(X,g)$ and $(X,g_j)$ are isomorphic over $S$ by Lemma \ref{lem:isom_MoriChamber}, and we obtain $g \circ g_j^{-1}\in \Aut(X/S)$. In particular, we have $(g\circ g_j^{-1})^* D \in \Nef{X/S}$.
        Moreover, $g$ and $g_j$ being isomorphisms in codimension one, 
        we have
        \[
        (g\circ g_j^{-1})^* D = (g_j^{-1})^* g^* D \in (g_j^{-1})^* \Pi.
        \]
        Thus we have  $(g\circ g_j^{-1})^* D \in \Sigma$ and finally
        conclude that $\Amp{X/S}\subset \Aut(X/S)\cdot \Sigma$.     
    \end{proof}

    \begin{corollary}
        Let $\pi\colon (X,B) \to S$ be a klt $K$-trivial fibration as in the Set-up \ref{setup-relcone}.
     Assume that for a very general point $s\in S$, 
good minimal models
exist for effective klt pairs on the fibre $X_s$.           
     If $\Mov{X/S}$ is non-degenerate and
the action of $\PsAut(X/S, B)$ on $\Move{X/S}$ admits a rational polyhedral fundamental domain,  then the following statements hold.
        \begin{enumerate}[(a)]
        \item Up to isomorphism over $S$, there exist only finitely many $X'$ arising as SQM of $X$ over $S$.
        \item The relative nef cone conjecture holds for each klt $K$-trivial fibration $X'\to S$ obtained by an SQM over $S$ from $X$.
        \end{enumerate}
    \end{corollary}
    \begin{proof}
        By hypothesis we have a rational polyhedral cone $\Pi\subset \Eff{X/S}$ such that the inclusion (\ref{eq:presquebir}) is satisfied. 
        So the statement (a) follows from Proposition \ref{prop:mov_implies_nef}.
    \end{proof}

    We now state a result which is essentially due to \cite{Kaw97}.

    \begin{proposition}\label{prop:decop_eff_and_mov}
    Let $\pi\colon (X, B) \to S$ be a klt $K$-trivial fibration. Then we have
    \begin{enumerate}[(A)]
    \item the inclusions
    \[
    \Bigcone{X/S}\subset\bigcup_{(Y,\mu)\text{}} \Eff {X/S,\mu} \subset \Eff{X/S},
    \]
    where the indices $(Y,\mu)$ range over all isomorphism classes of $\Q$-factorial birational contractions $\mu \colon X\dashrightarrow Y$
 over $S$ and the cones in the union have disjoint interiors.
    \item the inclusions
    \[
    \MovO{X/S}\subset\bigcup_{(Z,\alpha)\text{}} \Eff {X/S,\alpha} \subset \Move{X/S},
    \]
    where the indices $(Z,\alpha)$ range over all isomorphism classes of SQM 
$\alpha \colon X\dashrightarrow Z$ over $S$ and the cones in the union have disjoint interiors.
    \end{enumerate}

    Moreover, 
  assuming the existence of good minimal models
for effective klt pairs on the fibre $X_s$ for a very general point $s\in S$,          
 we have the equalities
    \begin{enumerate}[(a)]
        \item 
        \[
    \bigcup_{(Y,\mu)\text{}} \Eff {X/S,\mu} = \Eff{X/S},
    \]
    \item and
        \[
    \bigcup_{(Z,\alpha)\text{}} \Eff {X/S,\alpha} = \Move{X/S}.
    \]
    \end{enumerate}
    where the indices $(Y,\mu)$ and $(Z,\alpha)$ are as above. 
\end{proposition}

\begin{proof}
    The cones in the union have disjoint interiors by Lemma \ref{lem:isom_MoriChamber}.

    Let $D\subset X$ be a $\pi$-effective $\R$-divisor. Then there exists an effective $\R$-divisor in the same class as $D$ in $N^1(X/S)_{\R}$ by Lemma \ref{lem:rel_to_abs}(a). Therefore, we may assume that $D$ is effective.
        Since $X$ is klt, by replacing $D$ with $\epsilon D$ for $\epsilon >0$ sufficiently small, we may assume that the pair $(X,D)$ is klt. 
    
  {\em Proof of (A).} We first consider the case when $D$ is $\pi$-big. Since $(K_{X_s}, B_s)$ is trivial and $D_s$ is big, the pair $(X_s, B_s+D_s)$ has a good minimal
  model by \cite[Theorem 1.2]{BCHM10}. Thus by Theorem \ref{thm:HX13_R_divisor} we have a minimal model 
  $$
  \mu\colon (X, B+D)\dashrightarrow \big(Y, D_Y \coloneqq \mu_*(B+D)\big)$$ 
  over $S$. 
    Hence,
    \[
    D \equiv_\pi K_X+B+ D \equiv_\pi \mu^*(K_Y+D_Y)+\sum_i a_i E_i,
    \]
    where $K_Y+D_Y$ is nef over $S$ and the $E_i$'s run through all the prime divisor contracted by $\mu$ and $a_i> 0$. Note that $K_Y+D_Y$ is also $\pi$-effective, as $D$ is effective.
    Hence, we have
    \[
    \Bigcone{X/S}\subset\bigcup_{(\mu, Y)\text{}} \Eff {X/S, \mu},
    \]
where the indices $(Y,\mu)$ range over all isomorphism classes of $\Q$-factorial birational contractions $\mu \colon X\dashrightarrow Y$ over $S$.

 {\em Proof of (a).} Now assume the existence of good minimal models
for effective klt pairs on the fibre $X_s$.
 
Then the proof of $(A)$ applies for any $\pi$-effective divisor $D$ on $X$.
This gives
    \[
    \Eff{X/S}\subset\bigcup_{(\mu, Y)\text{}} \Eff{X/S,\mu},
    \]
    with the same set of indices and thus we obtain the equality.

{\em Proof of (B).} 
Note that for any small birational map $\alpha\colon X\dashrightarrow Z$, we have 
\begin{equation}
\label{easyinclusion}
\Eff{X/S,\alpha} = \alpha^*\Nefe{Z/S}\subset \Move{X/S}.
\end{equation}
  Since $\MovO{X/S}\subset \Bigcone{X/S}$, the inclusion 
    \[
    \MovO{X/S}\subset\bigcup_{(Z,\alpha)\text{}} \Eff{X/S,\mu} = \bigcup_{(Z,\alpha)\text{}} \alpha^*\Nefe{Z/S}
    \] follows from \cite[Theorem 2.3]{Kaw97}, taking \cite[Theorem 1.2]{BCHM10} and Theorem \ref{thm:HX13_R_divisor} into account. 
Combined with \eqref{easyinclusion} this gives the statement.

{\em Proof of (b).} 
Now assume the existence of good minimal models
for effective klt pairs on the fibre $X_s$. Then using again Theorem \ref{thm:HX13_R_divisor} we  obtain the existence of a minimal model  over $S$
for any klt pair $(X, B+D)$.

    Hence, again by \cite[Theorem 2.3]{Kaw97} we obtain the inclusion
    \[
    \Move{X/S}\subset\bigcup_{(Z,\alpha)\text{}} \Eff{X/S,\mu} = \bigcup_{(Z,\alpha)\text{}} \alpha^*\Nefe{Z/S},
    \] 
    which yields the equality.
\end{proof}

    \begin{proposition}\label{prop:nef_implies_mov}
        Let $\pi\colon (X, B) \to S$ be a klt $K$-trivial fibration as in the Set-up \ref{setup-relcone}.
             Assume that for a very general point $s\in S$, 
good minimal models
exist for effective klt pairs on the fibre $X_s$.              
Assume also that the following statements hold:
        \begin{enumerate}[(a)]
        \item up to isomorphism over $S$, there exist only finitely many $X'$ arising as SQM of $X$ over $S$;
        \item the relative nef cone conjecture holds for each klt $K$-trivial fibration $X'\to S$ obtained by an SQM over $S$ from $X$.
        \end{enumerate}
Then the action of $\PsAut(X/S, B)$ on $\Move{X/S}$ admits a rational polyhedral fundamental domain.
    \end{proposition}

    \begin{proof}
        We fix $\alpha_i\colon X\dashrightarrow X_i$ for $1\leq i\leq m$ a set of representatives of the finitely many isomorphism classes of SQM of $X$ over $S$. By  assumption $(b)$, for any $i\in\{1,\dots,m\}$ there exists a rational polyhedral cone $\Pi_i \subset \Nefe{X_i/S}$ such that
        \begin{equation}\label{eq:assm_nef_conj}
            \Aut(X_i/S, B)\cdot \Pi_i = \Nefe{X_i/S}.
        \end{equation}

        Set 
        $$
        \Pi \coloneqq \textrm{conv}\{\alpha_i^* \Pi_i\mid 1\leq i\leq m\} \subset \Move{X/S}.
        $$ 
        Then $\Pi$ is a rational polyhedral cone as the convex hull of finitely many rational polyhedral cones. By  Lemma \ref{lem:otherinc} and Lemma \ref{lem:loo} it suffices to show
        \[
        \Move{X/S}\subset\PsAut(X/S, B)\cdot \Pi.
        \]
        Now by Proposition \ref{prop:decop_eff_and_mov}(b), it suffices to show that for any SQM $\alpha\colon X\dashrightarrow Z$ over $S$, we have 
        $$
\Eff{X/S, \alpha} =  \alpha^*\Nefe{X/S}\subset \PsAut(X/S, B)\cdot \Pi.
        $$
        By assumption $(a)$ there exists $i_0\in\{1,\dots,m\}$ such that we have an isomorphism $f\colon Z \to X_{i_0}$ over $S$. Then
        \begin{align*}
            \alpha^*\Nefe{Z/S} &= \alpha^* f^* \Nefe{X_{i_0}/S} = \alpha_{i_0}^* \Nefe{X_{i_0}/S}
            \\
            &= \alpha_{i_0}^* \big(\Aut(X_{i_0}/S, B)\cdot \Pi_{i_0}\big)
                    \end{align*} 
where the last equality is given by \eqref{eq:assm_nef_conj}.
Now observe that for any $g \in \Aut(X_{i_0}/S, B)$ we have
$$
\alpha_{i_0}^* g^* = \alpha_{i_0}^* g^* (\alpha_{i_0}^{-1})^* \alpha_{i_0}^*
= (\alpha_{i_0}^{-1} \circ g \circ \alpha_{i_0})^* \alpha_{i_0}^*
$$
and $\alpha_{i_0}^{-1} \circ g \circ \alpha_{i_0} \in \PsAut(X/S, B)$.
Thus we have
\[
\alpha_{i_0}^* \big(\Aut(X_{i_0}/S, B)\cdot \Pi_{i_0}\big)
= \PsAut(X/S, B) \alpha_{i_0}^* \Pi_{i_0} \subset \PsAut(X/S, B) \cdot \Pi.\qedhere
\]
\end{proof}

\section{Fibrations in IHS manifolds}
\label{section:HK}
In this section we treat the case of fibrations in IHS manifolds for which unconditional results can be obtained. 
We record the following basic fact. 

\begin{lemma} \label{lem:fait}
    Let $X$ be a $\Q$-factorial variety, $\pi\colon X\to S$  a fibration, $K\coloneqq\mathbb C(S)$ the function field of the base and $j\colon X_K\hookrightarrow X$ the natural inclusion. We have a surjective map
    \[
    r\colon N^1(X/S)_{\R}\to N^1(X_K)_{\R},
    \]
    which induces the surjective maps
    \[
     \Eff{X/S} \to \Eff{X_K} \text{ and } \Mov{X/S} \to \Mov{X_K}.
    \]
\end{lemma}
\begin{proof}
    For any open set $U\subset S$ we have a natural surjective linear map: \[N^1(X/S)_{\R}\to N^1(X_U/U)_{\R},\ [D]\mapsto[D_{|U}].\] By \cite[Lemma 3.7]{Li23} for $U$ sufficiently small, we have $N^1(X_U/U)_{\R}\simeq N^1(X_K)_{\R}$ and thus the first statement follows. The surjections between the cones are proved in Step 1 of the proof of \cite[Theorem 1.3]{Li23}.
\end{proof}

\begin{remark*}
Since $R^1 \pi_* \sO_X=0$ (cf.\ Remark \ref{rem:setup}) we can apply
\cite[Corollary 5.14]{Kle05} to obtain $N^1(X_K)_{\Q}\simeq \pic(X_K)_{\Q}$. 
\end{remark*}

\begin{definition}
Let $K$ be a field of characteristic 0. Let $F$ be a smooth
projective variety over $K$ of even dimension $2n$. We say that $F$ is an irreducible symplectic manifold if 
\begin{itemize}
\item[-] $\pi_1^{\mbox{\tiny \'et}}(F_{\bar K})=\{1\}$,
\item[-] $H^0(F,\Omega^2_{F}) \simeq K$, and
\item[-] for $0 \neq \omega_K \in H^0(F,\Omega^2_{F})$ the exterior power
$\omega_K^{\wedge n}$ does not vanish. 
\end{itemize}
\end{definition}

\begin{lemma}\label{lem:cadorel}
Let $X$ be a $\Q$-factorial variety, $\pi:X\to S$  a fibration, $K\coloneqq\mathbb C(S)$ the function field of the base and $j\colon X_K\hookrightarrow X$ the natural inclusion. 
If the very general fibre of $\pi$ is simply connected, then $X_{\overline K}$ is simply connected. 
\end{lemma}
\begin{proof}
Let $f_{\overline{K}}\colon Y_{\overline{K}} \to X_{\overline{K}}$ be an \'etale morphism. By taking an open subset $U\subset S$ and an \'etale cover of it $U'\to U$, 
the induced \'etale morphism $f_{U'}\colon Y_{U'}\to X_{U'}$ is such that $f_{\overline{K}} $ is the base change of $f_{U'}$ given by an algebraic field extension of $\overline K$. Since $\overline K$ is algebraically closed, the extension splits and
therefore $Y_{\overline{K}}$ is a disjoint union of copies of $X_{\overline{K}}$.
  \end{proof}

\begin{lemma}\label{lem:IHS}
Let $X$ be a $\Q$-factorial variety, $\pi \colon X\to S$  a fibration, $K\coloneqq\mathbb C(S)$ the function field of the base and $j\colon X_K\hookrightarrow X$ the natural inclusion. 
If the very general fibre of $\pi$ is a projective IHS manifold, then $X_K$ is an irreducible symplectic variety over $K$. 
\end{lemma}
\begin{proof}

%
%
Let $s\in S$ be a very general point.
Since $X_K \coloneqq X_{\mbox{Spec}\big( \C(S)\big)}$ is the base change of the fibration 
$\pi \colon X \rightarrow S$ to $\mbox{Spec}\big( \C(S)\big)$, we have $\Omega_{X_K}^2 = j^* \Omega^2_{X/S}$
and
$$
H^0(X_K,\Omega^2_K) = H^0 (\mbox{Spec}\big(\C(S)\big), (\pi_K)_* \Omega_{X_K}^2),
$$
where $\pi_K = \pi \circ j$. 
Yet $(\pi_K)_* \Omega_{X_K}^2 \simeq (\pi_K)_* j^* \Omega^2_{X/S}
\simeq (\pi_* \Omega^2_{X/S}) \otimes \sO_{\mbox{Spec}\big( \C(S)\big)}$,
so $h^0 \big(\mbox{Spec}\big( \C(S)\big), (\pi_K)_* \Omega_{X_K}^2 \big)$ is just the rank of the direct image sheaf $\pi_* \Omega^2_{X/S}$; since $X_s$ is an IHS manifold, the rank is equal to one.

Choose now $0 \neq \omega_K \in H^0(X_K,\Omega^2_K)$, then $\omega \coloneqq j_* \omega_K$ induces a symplectic form on $X_s$. The locus $S_0 \subset S$ (not necessarily closed) of points $p$ such that ${\omega^{\wedge n}}_{|X_p}$ does not vanish, where $n\coloneqq \frac{\dim X_p}{2}$, is Zariski open. Since $\omega$ is symplectic on $X_s$, the set $S_0$ is not empty and therefore contains
$\mbox{Spec}\big( \C(S) \big)$.

Finally we know by Lemma \ref{lem:cadorel} that $X_{\overline K}$ is simply connected and we are done. 
\end{proof}

\begin{proposition}\label{prop:nonvan}
Let $\pi\colon X\to S$  be a fibration between $\Q$-factorial varieties with very general fibre $F$. Let $K\coloneqq\mathbb C(S)$ be the function field of the base and $j\colon X_K\hookrightarrow X$ the natural inclusion. 
\begin{enumerate}[(a)]
\item $\Nefp{F}\subset \Nefe{F}$ implies $\Nefp{X_K}\subset \Nefe{X_K}$.
\item If the very general fibre of $\pi$ is a projective IHS manifold, where every nef divisor is semiample, we have:
\subitem (b1) $\Movp{X_K}\subset \Move{X_K}$;
 \subitem(b2) $\Movp{X/S}\subset \Move{X/S}$.
\item If the very general fibre of $\pi$ is a projective IHS manifold, we have $\Movp{X_K}\supset \Move{X_K}$. 
\end{enumerate}
\end{proposition}
\begin{proof}
$(a)$ Consider $D_K\in \Nefp{X_K}$. Without loss of generality we may
assume that $D_K$ is rational. By Lemma \ref{lem:fait} there exists $D\in N^1(X)_\Q$ such that $D_{|X_K}=D_K$. Let $H$ be an ample divisor on $X$. 
For any integer $m>0$ the divisor 
\begin{equation}\label{eq:amp}
(D+\frac{1}{m}H)_{|X_K} \textrm{ is ample.} 
\end{equation}
On the other hand, we know (cf.\ \cite[Corollary 9.6.4]{EGAIV}) that for each $m$ there exists an open subset $U_m\subset S$ such that $(D+\frac{1}{m}H)_{|X_s}$ is ample for all $s\in U_m$, and $U_m$ is not empty by (\ref{eq:amp}). Now set $U\coloneqq\cap_{m\in \N^*} U_m$. Then $D_{X_s}$ is nef for all $s\in U$.
By the hypothesis it follows that $D_{X_s}$ is effective  for any $s\in U$. Hence $D\in \Eff{X/S}$ and $D_K\in \Eff{X_K}$.

$(b)$ We give a detailed proof of $(b1)$. The proof of $(b2)$ is exactly the
same. Consider $D_K\in \Movp{X_K}$. Without loss of generality we may
assume that $D_K$ is rational.
By Lemma \ref{lem:fait} there exists $D\in N^1(X)_\Q$ such that $D_{|X_K}=D_K$. Let $H$ be an ample divisor on $X$. Then for any integer $m>0$ the divisor 
$(D+\frac{1}{m}H)_{|X_K} $ is big.  As such  we write it as 
$$
(D+\frac{1}{m}H)_{|X_K} = A_{m,K} + E_{m,K}
$$
with $A_{m,K} \in \Amp{X_K}$ and  $ E_{m,K}\in \Eff{X_K}$. Let $A_{m} \in N^1(X)_{\mathbb Q}$ and  $ E_{m}\in N^1(X)_{\mathbb Q}$ such that $(A_m)_{|X_K}=A_{m,K}$ and $ (E_m)_{|X_K}=E_{m,K}$. As before, for each $m$ there exists a non-empty open subset $U_m\subset S$ such that $(D+\frac{1}{m}H)_{|X_s}=(A_m+E_m)_{|X_s}$ is big for all $s\in U_m$. Setting again $U\coloneqq\cap_m U_m$ we deduce that $D_{|X_s}$ is pseudoeffective for all $s\in U$. Now consider the divisorial Zariski decomposition $D_{|X_s}=P+N$ (see \cite[Corollary 4.11]{Bou04}). The positive part $P$ is movable and more precisely lies in $\Movp{X_s}$, by construction as $D_K\in\Movp{X_K}$. By \cite[Lemma 2.18]{Den22} we have $\Movp{X_s}\subset \Move{X_s}$. Notice that \cite[Lemma 2.18]{Den22} is stated and proved for projective IHS manifolds  of the 4 known deformation types, but what is used is that on the IHS every nef divisor is semiample. In particular, $P$ is effective, thus so is $D_{|X_s}$. Hence, $D$ is $\pi$-effective, and therefore $D_K$ is effective. 

%
$(c)$ Consider $D_K\in \Move{X_K}$.  By Lemma \ref{lem:fait} there exists $D\in\Move{X/S}$ such that $D_{|X_K}=D_K$. By Lemma \ref{lem:otherinc} we have $D\in \Movp{X/S}$, hence $D_K\in \Movp{X_K}$.
 \end{proof}

We are now ready to prove our main result. 
\begin{proof}[Proof of Theorem \ref{thm:IHS-intro-gen}]
As before  $K\coloneqq\mathbb C(S)$ denotes the function field of the base and $j:X_K\hookrightarrow X$ the natural inclusion. Notice that $X_K$ is an irreducible symplectic variety over $K$ by Lemma \ref{lem:IHS}.  
By \cite[Theorem 1.0.5]{Tak21} the action of $\PsAut{X_K}$ on $\Movp{X_K}$ admits a rational polyhedral fundamental domain. By Proposition \ref{prop:nonvan}(c)  we have $\Movp{X_K}\supset \Move{X_K}$ 

By \cite[Proof of Theorem 6.1]{LZ22} the weak cone conjecture for $\Move{X_K}$ implies the relative weak cone conjecture for $\Move{X/S}$. Now observe that by Lemma \ref{lem:otherinc} and Proposition \ref{prop:nonvan}(b2), we have $\Move{X/S}=\Movp{X/S}$. 

Hence we can apply Looijenga's result Lemma \ref{lem:loo} to deduce the  relative  cone conjecture for $\Move{X/S}$.
From Proposition \ref{prop:mov_implies_nef}(a), we deduce that, up to isomorphism over $S$, there exist only finitely many $X'$ arising as SQM of $X$ over $S$. From Proposition \ref{prop:mov_implies_nef}(b), we deduce that for each of them the relative Nef cone conjecture holds. 
\end{proof}
\begin{proof}[Proof of Corollary \ref{cor:IHS}] It follows immediately from Theorem \ref{thm:IHS-intro-gen} and Remark \ref{rmk:gmm}(c).
\end{proof}

\section{A family of K3 surfaces}
\label{sectionfamilyK3}

We consider a family of K3 surfaces obtained as a refinement of the construction in \cite[Section 3]{BKPS98}:
let $\holom{\psi}{A}{C}$ be a general complete family of abelian surfaces
given by a curve $C \subset \bar{\mathcal A}_2$ cut out by general sufficiently ample divisors. Since the boundary of the Satake compactification has codimension more than one 
all the $\psi$-fibres are smooth and the very general fibre is an abelian surface with Picard number one. Up to making a base change we can assume that the family has a section $\holom{s_1}{C}{A}$ which we consider as the zero section giving $\psi$ the structure of an abelian group scheme.   The two-torsion points of this abelian group scheme form an \'etale multisection
$T_2 \subset A$ of degree 16 with $s_1(C) \subset T_2$ being a connected component.
A priori the other connected components of $T_2$ are not $\psi$-sections, but up to finite  base change we can assume that
$$
T_2 = \bigcup_{i=1}^{16} s_i(C)
$$
with $s_i\colon C \rightarrow A$ a section for any $i\in\{1,\dots,16\}$. Again, up to finite base change, we can assume that $g(C)>1$. The automorphism $z \mapsto -z$ acts on every $\psi$-fibre, so we can 
apply the Kummer construction to this family to obtain a smooth family of K3 surfaces of Kummer type
\begin{equation}
\label{thefamily}
\holom{\pi}{X}{C}
\end{equation}
such that the very general fibre $F$ has
Picard number $17$. 
More precisely, we denote by
$\holom{q}{A}{A/\pm 1}$ the quotient map and by 
$$
\holom{\mu}{X}{A/\pm 1}
$$
the fibrewise minimal resolution. If $\Theta \rightarrow A$ is a relative principal polarisation, there exists a Cartier divisor $H \rightarrow A/\pm 1$ such that
$q^* H \simeq 2 \Theta$. 

Note that $q\big(s_i(C)\big)$ is a section
of $A/\pm 1 \rightarrow C$ and $\mu$ is the blowup along the 16 disjoint curves $q\big(s_i(C)\big)$.
Thus we have 16 disjoint prime divisors
$$
E_1, \ldots, E_{16} \subset X
$$
that are $\mu$-exceptional and restrict to a $(-2)$-curve on every fibre.

An intersection computation now shows that the classes $\mu^* H, E_1, \ldots, E_{16}$ 
are linearly independent in $N^1(X/C)$, so $N_1(X/C)$ has rank 17 and for a very general fibre the restriction morphism
$$
N^1(X) \rightarrow N^1(F)
$$
is surjective. Note also that by the canonical bundle formula $K_{X/C} \simeq \pi^* M$, where $M$ is an ample line bundle on $C$.

\subsection{The relative automorphism group}

\begin{lemma} \label{lemmaEnegative}
Let $E_i$ be one of the exceptional divisors from the Kummer construction. Then we have $E_i \cdot c_2(X)=- 2 \deg M<0$.
\end{lemma}

\begin{proof}
In order to simplify the notation we will denote $E_i$ by $E$.
Consider the exact sequence
$$
0 \rightarrow T_{E} \rightarrow {T_X}_{|E} \rightarrow N_{E/X} \rightarrow 0.
$$
Then we have
$$
E \cdot c_2(X) = c_2({T_X}_{|E}) = c_2(E) + c_1(E) \cdot c_1(N_{E/X}),
$$
so we only have to compute these two terms.
For the ruled surface $E$ we have 
$$
c_2(E) = e (E) = - 2 \deg K_C
$$
and $K_{E/C}^2=0$.
By adjunction and $E^2 \cdot F=-2$ we have
$$
K_{E/C}^2 = (K_{X/C}+E)^2 \cdot E = (\pi^* M+E)^2 \cdot E = - 4 \deg M + E^3,
$$
therefore $E^3=4 \deg M$. Again by adjunction we have
\begin{align*}
    c_1(E) \cdot c_1(N_{E/X}) &= -K_E \cdot E_{|E} = - (K_X+E) \cdot E^2  \\
& = - \big(\pi^* (K_C+M) + E\big) \cdot E^2 \\
& = 2 \deg K_C  + 2 \deg M - E^3
= 2 \deg K_C - 2 \deg M.\qedhere
\end{align*}
\end{proof}

Since the threefold $X$ is not uniruled we know by Miyaoka's theorem \cite[Theorem 1.1]{Mi87}
that $c_2(X) \in \NEX$, so we expect its intersection with most prime divisors 
to be non-negative. In our case this can be made more precise:

\begin{proposition}
\label{proposition-negative-surfaces}
Let $S \subset X$ be a prime divisor that is distinct from the exceptional divisors $E_j$. Then we have $S \cdot c_2(X) \geq 0$.
\end{proposition}

The set-up of the proof requires some preparation: let 
\begin{equation}
\label{sequence-psi}
0 \rightarrow \psi^* \Omega_C \rightarrow \Omega_A \rightarrow \Omega_{A/C} \rightarrow 0
\end{equation}
be the cotangent sequence for $\psi$. The relative cotangent bundle is trivial on every fibre, so $\Omega_{A/C} \simeq \psi^* V$
with $V$ a rank two vector bundle on the curve $C$. Using Miyaoka's semipositivity
theorem  \cite[Corollary 6.4]{Mi87} it is not hard to see that $V$ is nef.
Moreover since we started the construction
with a sufficiently positive curve $C \subset \bar{\mathcal A}_2$, 
we can deduce from the stability of the Hodge bundle \cite[Proposition 0.1]{MVZ12}
that $V$ is semistable. Thus we obtain:

\begin{proposition} \label{propositionVsemistable} 
The vector bundle $V$ is nef and semistable.
\end{proposition}

Let $\holom{\tilde \mu}{Y}{A}$ be the blowup along the 16 disjoint sections $s_i(C)$
by $D_i \subset Y$ the exceptional divisors. By construction of the blowup
$D_i = \PP(N_{s_i(C)/A}^*)$  and since $s_i(C)$ is a $\psi$-section we have
$$
N_{s_i(C)/A}^* \simeq \Omega_{A/C} \otimes \sO_{s_i(C)} \simeq V.
$$
Let
$\holom{\eta}{Y}{X}$ the double cover branched over the 16 divisors $E_i \simeq D_i$. The double cover $\eta$ is cyclic, so there exists a (non-effective) Cartier divisor class $\delta$ on $X$ such that
$$
2 \delta \simeq \sum_{i=1}^{16} E_i.
$$

\begin{lemma} \label{lemmactwo}
We have
$$
c_2(X) = -3 \delta^2 - \pi^* c_1(V) \delta.
$$
\end{lemma}

\begin{proof}
From the exact sequence \eqref{sequence-psi} we obtain immediately that
$K_A \simeq \psi^* (K_C+ \det V)$ and $c_2(A)=0$.

Consider the cotangent sequence for the blowup $\tilde \mu$:
\begin{equation}
\label{sequence-blowup}
0 \rightarrow \tilde \mu^* \Omega_A \rightarrow \Omega_Y \rightarrow \Omega_{Y/A} \simeq \oplus_{i=1}^{16} \Omega_{D_i/C_i} \rightarrow 0. 
\end{equation}
Recall that $
{-D_i}_{|D_i} \simeq \sO_{\PP(N^*_{C_i/A})}(1) \simeq \sO_{\PP(V)}(1)$.
Denote by $\holom{\mu_i}{D_i}{C}$ the ruling, then we have 
$$
\Omega_{D_i/C_i} 
\simeq
K_{D_i/C_i} \simeq \mu_i^* \det V + 2 {D_i}_{|D_i}.
$$
Since this invertible sheaf is supported on the prime divisor $D_i$, we have
$$
c_1(\Omega_{D_i/C_i})=D_i, \qquad c_2(\Omega_{D_i/C_i})=-D_i^2- \tilde \mu^* \psi^* c_1(V) \cdot D_i.
$$
From \eqref{sequence-blowup} we obtain
$$
K_Y \simeq \tilde \mu^* \psi^* \big(K_C+c_1(V)\big) + \sum_{i=1}^{16} D_i,
\qquad
c_2(Y) = \mu^* \psi^* K_C \cdot \sum_{i=1}^{16} D_i - \sum_{i=1}^{16} D_i^2.
$$

Now consider the cotangent sequence for the double cover $\eta$:
\begin{equation}
\label{sequence-cotangent-eta}
0 \rightarrow \eta^* \Omega_X \rightarrow \Omega_Y \rightarrow \Omega_{Y/X} \simeq \oplus_{i=1}^{16} 
\sO_{D_i}(-D_i) \rightarrow 0
\end{equation}
It is elementary to compute that
$$
c_1\big( \sO_{D_i}(-D_i)\big) =D_i \text{ and }
c_2\big( \sO_{D_i}(-D_i)\big)= 2 D_i^2.
$$
Therefore, we deduce from \eqref{sequence-cotangent-eta} that
$$
c_1(\eta^* \Omega_X) = \tilde \mu^* \psi^* \big( K_C+c_1(V)\big)
\text{ and }
c_2(\eta^* \Omega_X) = -3 \left(\sum_{i=1}^{16} D_i\right)^2 - \tilde \mu^* \psi^* c_1(V) \sum_{i=1}^{16} D_i.
$$
Now recall that $\eta^* \sum_{i=1}^{16} E_i = 2 \sum_{i=1}^{16} D_i$ and therefore $\sum_{i=1}^{16} D_i \equiv \eta^* \delta$.
Hence,
$$
\eta^* c_2(X) = 
c_2(\eta^* \Omega_X) = -3 \eta^* \delta^2 - \eta^* \pi^* c_1(V) \delta 
$$
and the injectivity of $\holom{\eta^*}{H^4(X, \C)}{H^4(Y, \C)}$
allows to conclude.
\end{proof}

\begin{proof}[Proof of Proposition \ref{proposition-negative-surfaces}]
Recall that 
$$
E_i \simeq D_i \simeq \PP(V),
$$
where $V \rightarrow C$ is a semistable vector bundle of rank two (see Proposition \ref{propositionVsemistable}).
Thus if we denote by $\holom{\mu_i}{E_i}{C}$ the ruling and by
$\zeta_i \coloneqq c_1\big(\sO_{\PP(V)}(1)\big) \rightarrow E_i$ the tautological class,
then the class $2 \zeta_i - \mu_i^* c_1(V)$ is nef \cite[Proposition 6.4.11]{Laz04b}.
Recall also that $2 \delta \simeq \sum_{i=1}^{16} E_i$, 
so $\delta_{|E_i} \simeq -\zeta_i$ for every $i=1, \ldots, 16$.

Let now $S \subset X$ be a prime divisor that is distinct from the $E_i$. 
Then by Lemma \ref{lemmactwo} we obtain
\begin{multline*}
S \cdot c_2(X) = S \cdot \big( -3 \delta^2- \pi^* c_1(V) \delta\big)
= \frac{1}{2} S \cdot \big( -3 \delta-\pi^* c_1(V)\big) \cdot \sum_{i=1}^{16} E_i
\\
= \frac{1}{2} \sum_{i=1}^{16} S_{|E_i} \cdot \big( -3 \delta-\pi^* c_1(V)\big)_{|E_i} 
= \frac{1}{2} \sum_{i=1}^{16} (S \cap E_i) \cdot \big( 3 \zeta_i-\mu_i^* c_1(V)\big)
\end{multline*}
Since $2 \zeta_i - \mu_i^* c_1(V)$ is nef and $\zeta_i$ is nef by
Proposition \ref{propositionVsemistable},
the class $3 \zeta_i-\mu_i^* c_1(V)$ is nef. Yet $S \cap E_i$ is an effective divisor on $E_i$, so the statement follows.
 \end{proof}

\begin{proof}[Proof of Proposition \ref{proposition:aut:finite}]
Since $g(C)>1$ the automorphism group $\Aut(C)$ is finite.
Moreover, the $\pi$-fibres $F$ having $q(F)=0$, it is clear that 
every $g \in \Aut(X)$ satisfies $g^*F = F$
Let $E_j$ be one of the $\mu$-exceptional divisors. By Lemma \ref{lemmaEnegative}
and $g^* c_2(X)=c_2(X)$ we have
$$
g^* E_j \cdot c_2(X) = g^* E_j \cdot g^* c_2(X) = g^* \big( E_j \cdot  c_2(X)\big) < 0. 
$$
Hence, by Proposition \ref{proposition-negative-surfaces} the surface $g^* E_j$ is one of the exceptional divisor $E_i \subset X$. In other words, we have an action on the set of exceptional divisors giving a group morphism
$$
\Aut(X) \rightarrow S_{16}.
$$
Denote by $K$ the kernel of this map. 
For every $g \in K$ we have  $g^* E_j= E_j$. Since we also have $g^* F=F$,
an intersection computation shows that $g^* \mu^* H = \lambda \mu^* H$ for some $\lambda \in \R^+ $. Yet $g^*$ is an automorphism of the $\Z$-module $\NS(X)$, so $\lambda=1$.
Thus the group $K$ acts trivially on $N^1(X)$, in particular it fixes a polarisation.
This implies that $K$ is a finite group. In particular $\Aut(X)$, hence $\Aut(X/C)$, is finite.
\end{proof}

\subsection{An interesting SQM} \label{subsectionSQM}

\begin{lemma} \label{lemma-no-monodromy}
Let $F$ be a very general fibre of the family of Kummer surfaces as in \eqref{thefamily}.
Let $B \subset F$ be any $(-2)$-curve. Then there exists a prime divisor
$D \subset X$ such that we have an equality of cycles $D \cdot F = B$.
\end{lemma}

It is well-known that this follows from the surjection $N^1(X) \rightarrow N^1(F)$, we give the details for the convenience of the reader:

\begin{proof}
Since $F$ is very general we know by countability of the Hilbert scheme that the curve $B \subset F$ deforms with $F$ and it is clear that a small deformation is a $(-2)$-curve.
Thus the cycle space $\Chow{X}$ is irreducible and of dimension one in the point $[B]$ and the image of the universal family defines a prime divisor $D \subset X$ such that
$B \subset D$ is an irreducible component of a fibre of $D \rightarrow C$.
We claim that the general fibre of $D \rightarrow C$ is irreducible, which implies the statement. For the proof of the claim note that up to replacing $F$ by some other very general fibre, we can assume without loss of generality that $D \cap F$ is reduced and every irreducible component is a deformation of $B$. Now recall that
the restriction map $N^1(X) \rightarrow N^1(F)$ is surjective, so there are uniquely defined $\alpha, \beta_i \in \Q$ such that
$$
[B] = (\alpha H + \sum \beta_i E_i)_{|F} 
$$
in $N^1(F)$. Thus the class of $B$ (and all its deformations) in $N_1(X)$ is $(\alpha H + \sum \beta_i E_i) \cdot F$. Yet this implies that two irreducible components
$B, B' \subset D \cap F$ have the same class in $F$. Since $B$ is a $(-2)$-curve, we get
$B=B'$.
\end{proof}

Lemma \ref{lemma-no-monodromy} shows that the deformations of every $(-2)$-curve in a very general fibres cover a divisor that defines an extremal ray in the relative pseudoeffective cone $\overline{\textrm{Eff}} (X/C)$. We will now use our result and MMP to show that, after a suitably chosen SQM, they are actually extremal rays in the relative Mori cone.  

\begin{proposition}\label{prop:interesting}
Let $\holom{\pi}{X}{C}$ be the family of Kummer surfaces \eqref{thefamily}.
There exists an SQM $\psi: X \dashrightarrow \tilde X$ over $C$ such that
\begin{itemize}
\item[-] The variety admits infinitely divisorial contractions $\holom{\mu_j}{\tilde X}{Y_j}$
over $C$.
\item[-] The relative automorphism group $\Aut(\tilde X/C)$, and therefore 
$\Bir(X/C)$, is infinite.
\end{itemize}
\end{proposition}

\begin{proof}
The second statement is immediate from the first statement and Corollary \ref{cor:K3s}, since $\tilde X$ has infinitely many extremal rays in $\NE{X/C}$, the automorphism group must be infinite.

For the proof of the first statement, note first that by Corollary \ref{cor:IHS} we can choose finitely many SQM's $\merom{\psi_l}{X}{\tilde X_l}$ over $C$ such that every other
SQM is isomorphic to some $\tilde X_l$. Fix now a very general fibre $F \subset X$, and
let $(B_j)_{j \in \N} \subset F$ be an infinite collection of distinct $(-2)$-curves (which are known to exist on a Kummer surface). By Lemma \ref{lemma-no-monodromy} there
exist prime divisor $D_j \subset X$ such that $D_j \cdot F=B_j$ as cycles. Note that the divisors $D_j$ are distinct, since their intersections with $F$ are distinct curves.

For every $j \in \N$, fix $\epsilon_j>0$ such that the pair $(X, \epsilon_j D_j)$ is klt.
The divisor $K_X+\epsilon_j D_j$ is not $\pi$-nef, since the restriction to $F$ is not nef. Thus we can run a $(K_X+\epsilon_j D_j)$-MMP over $C$ decomposing into a SQM
$$
\merom{\varphi_j}{X}{X_j}
$$
and a divisorial contraction $\holom{\tau_j}{X_j}{Y_j}$ contracting
the divisor $(\varphi_j)_* D_j$ onto a curve in $Y_j$. 

By what precedes we know that up to renumbering, there are infinitely $j \in \N$ such
that $X_j \simeq \tilde X_1$ over $C$. Thus we can find isomorphisms
$$
\holom{f_j}{X_j}{\tilde X\coloneqq \tilde X_1}
$$
such that $f_j \circ \varphi_j = \psi_1 \eqcolon \psi$. We also set
$$
\holom{\mu_j\coloneqq\tau_j \circ f_j^{-1}}{\tilde X}{Y_j}
$$
and claim that the divisorial contractions $\mu_j$ are distinct. Yet $\mu_j$
contracts exactly the divisor $(f_j)_* (\varphi_j)_* D_j = \psi_* D_j$, so this is clear.
\end{proof}

\begin{remark}
The family $\tilde X \rightarrow C$ is still a smooth family of K3 surfaces of Kummer type, but it is no longer clear that it is globally Kummer, i.e.\ we have
a birational morphism
$$
\tilde \mu \colon \tilde X \rightarrow \tilde A/\pm 1.
$$
In fact the proof of Proposition \ref{proposition:aut:finite} indicates that such a structure contradicts the existence of infinitely many automorphisms. 

Recently Oguiso used his theory of $c_2$-contractions to study the relative automorphism group of fibered Calabi-Yau threefold \cite{Ogu24}.  
One might hope that a similar theory for families of K3 surfaces allows to get a better
understanding of $\Aut(\tilde X/C)$.
\end{remark}

\bibliographystyle{alpha}
\bibliography{biblio.bib}

\end{document}